\documentclass[11pt]{amsart}
\usepackage[backref=page]{hyperref}
\usepackage[top=2.7cm,left=2.8cm,right=2.8cm,bottom=2.6cm]{geometry}
\usepackage[usenames,dvipsnames]{xcolor}
\hypersetup{colorlinks=true,citecolor=NavyBlue,linkcolor=Brown,urlcolor=Orange}

\usepackage{amssymb, comment, paralist, xspace, graphicx, url, amscd, euscript, mathrsfs,stmaryrd,epic,eepic,color}
\usepackage{tikz-cd}
\usepackage[all]{xy}
\usepackage{amsthm}
\usepackage{amsmath}
\usepackage{amsfonts}
\usepackage{ dsfont }
\usepackage{enumerate}
\usepackage{xcolor}
\usepackage{tikz}
\usetikzlibrary{shapes,arrows,decorations.pathreplacing,backgrounds,positioning,fit,matrix}
\tikzstyle{block} = [rectangle, draw, fill=blue!20, 
    text width=5em, text centered, rounded corners, minimum height=4em]
\usepackage{float}

\SelectTips{cm}{}

\calclayout

1

\allowdisplaybreaks[1]

\newtheorem*{namedtheorem}{\theoremname}
\newcommand{\theoremname}{testing}
\newenvironment{named}[1]{\renewcommand\theoremname{#1}
\begin{namedtheorem}}
{\end{namedtheorem}}

\newtheorem{theorem}{Theorem}[section]
\newtheorem{proposition}[theorem]{Proposition}
\newtheorem{proposition-definition}[theorem]{Proposition-Definition}
\newtheorem{corollary}[theorem]{Corollary}
\newtheorem{lemma}[theorem]{Lemma}
\newtheorem{conjecture}[theorem]{Conjecture}

\theoremstyle{definition}
\newtheorem{definition}[theorem]{Definition}

\newtheorem{question}[theorem]{Question}

\newtheorem{example}[theorem]{Example}
\newtheorem{remark}[theorem]{Remark}

\theoremstyle{remark}

\newcommand\cA{\mathcal{A}}

\newcommand\cF{\mathcal{F}}
\newcommand\cG{\mathcal{G}}
\newcommand\cH{\mathcal{H}}

\newcommand\cL{\mathcal{L}}
\newcommand\cM{\mathcal{M}}

\newcommand\cO{\mathcal{O}}

\newcommand\cS{\mathcal{S}}

\newcommand\cX{\mathcal{X}}
\newcommand\cY{\mathcal{Y}}

\newcommand\CC{\mathbb{C}}

\newcommand\GG{\mathbb{G}}

\newcommand\NN{\mathbb{N}}

\newcommand\PP{\mathbb{P}}
\newcommand\QQ{\mathbb{Q}}

\newcommand\ZZ{\mathbb{Z}}

\newcommand\rD{\mathrm{D}}

\newcommand\fS{\mathfrak{S}}

\newcommand\fX{\mathfrak{X}}

\newcommand\frc{\mathfrak{c}}

\newcommand\frh{\mathfrak{h}}

\newcommand\frl{\mathfrak{l}}

\newcommand\rank{{\rm rank}}
\newcommand\cha{{\rm char}}
\newcommand\NS{{\rm NS}}
\newcommand\Nt{{\rm Nwt}}
\newcommand\Hdg{{\rm Hdg}}

\newcommand\Ext{{\rm Ext}}
\newcommand\cris{{\rm cris}}

\newcommand\srM{\mathscr{M}}
\newcommand\srS{\mathscr{S}}
\newcommand\arr{\ifinner\to\else\longrightarrow\fi}

\def\displaytimes_#1{\mathrel{\mathop{\times}\limits_{#1}}}

\def\displayotimes_#1{\mathrel{\mathop{\bigotimes}\limits_{#1}}}

\newcommand\alg{\operatorname{alg}}

\newcommand\End{\operatorname{End}}

\newcommand\Frac{\operatorname{Frac}}
\newcommand\Griff{\operatorname{Griff}}

\newcommand\Spec{\operatorname{Spec}}

\newcommand\Spf{\operatorname{Spf}}

\newcommand{\Pic}{\operatorname{Pic}}

\newcommand\rk{\operatorname{rk}}

\newcommand\id{\mathrm{id}}

\newdir{ >}{{}*!/-5pt/@{>}}

\newcommand\doublelong[2]{\mathbin{\xymatrix{{}\ar@<3pt>[r]^{#1}
\ar@<-3pt>[r]_{#2}&}}}

\newlength{\ignora}

\newcommand\Sym{\operatorname{Sym}}

\DeclareMathOperator \Alb{Alb}
\DeclareMathOperator \alb{alb}
\DeclareMathOperator \car{char}
\DeclareMathOperator \ch{ch}

\DeclareMathOperator \spe{sp}
\DeclareMathOperator \td{td}

\newcommand{\M}{\mathcal M}

\newcommand{\CH}{{\rm CH}}
\newcommand{\CHM}{{\rm CHM}}
\renewcommand{\ch}{{\rm ch}}
\newcommand{\h}{{\mathfrak h}}
\newcommand{\DCH}{{\rm DCH}}

\newcommand{\et}{{\rm\acute{e}t}}
\newcommand{\Br}{{\rm Br}}
\renewcommand{\1}{\mathop{\mathds{1}}\nolimits} 

\newcommand{\starM}{\star_{\mathrm{Mult}}}
\newcommand{\starC}{\star_{\mathrm{Chern}}}

\renewcommand{\tilde}{\widetilde}

\begin{document}

\title[Supersingular irreducible symplectic varieties]{Supersingular irreducible symplectic varieties}

\author{Lie Fu}
\address{Institut Camille Jordan, Universit\'e Claude Bernard Lyon 1,
	43 Boulevard du 11 novembre 1918,
	69622 Villeurbanne Cedex, France}
\address{Radboud University, Heyendaalseweg 135, 6525 AJ Nij\-megen, Netherlands}
\email{fu@math.univ-lyon1.fr}

\author{Zhiyuan Li}
\address{Shanghai Center for Mathematical Science\\
Fudan University\\
2005 Songhu Road\\
20438 Shanghai, China}
\email{zhiyuan\_li@fudan.edu.cn}

\date{\today}


\thanks{Lie Fu is supported by the Radboud Excellence Initiative program. He was also supported by the Agence Nationale de la Recherche through ECOVA (ANR-15-CE40-0002),  HodgeFun (ANR-16-CE40-0011), LABEX MILYON (ANR-10-LABX-0070) of Universit\'e de Lyon, and \emph{Projet Inter-Laboratoire} 2017,  2018 and 2019 by F\'ed\'eration de Recherche en Math\'ematiques Rh\^one-Alpes/Auvergne CNRS 3490. Zhiyuan Li is supported by National Science Fund for  General Program (11771086),   Key Program (11731004) and the ``Dawn" Program (17SG01) of Shanghai Education Commission}

\begin{abstract}
In complex geometry, irreducible holomorphic symplectic varieties, also known as compact hyper-K\"ahler varieties, are natural higher-dimensional generalizations of K3 surfaces. We propose to study such varieties defined over fields of positive characteristic, especially the supersingular ones, generalizing the theory of supersingular K3 surfaces.

 In this work, we are mainly interested in the following two types of symplectic varieties over an algebraically closed field of characteristic $p>0$, under natural numerical conditions:\\ (1) smooth moduli spaces of semistable (twisted) sheaves on K3 surfaces, \\(2) smooth Albanese fibers of moduli spaces of semistable sheaves on abelian surfaces.\\ Several natural definitions of the supersingularity for symplectic varieties are discussed, which are proved to be equivalent in both cases (1) and (2). Their equivalence is expected in general.
 
  On the geometric aspect, we conjecture that unirationality characterizes supersingularity for  symplectic varieties. Such an equivalence is established in case (1), assuming the same is true for K3 surfaces. In case (2), we show that  rational chain connectedness is equivalent to supersingularity. 
  
  On the motivic aspect, we conjecture that the algebraic cycles on supersingular symplectic varieties should be much simpler than their complex counterparts: its rational Chow motive is of supersingular abelian type, the rational Chow ring is representable and satisfies the Bloch--Beilinson conjecture and  Beauville's splitting property. As evidences, we prove all these conjectures on algebraic cycles for supersingular varieties in both cases (1) and (2). 
\end{abstract}

\maketitle

\setcounter{tocdepth}{2}


\section{Introduction}

This paper is an attempt to generalize to higher dimensions the beautiful theory  of supersingular K3 surfaces. 

\subsection{Supersingular K3 surfaces}
\label{subsection:SSK3}
Let $S$ be a K3 surface over an algebraically closed field  $k$ of characteristic $p >0$. 
On the one hand, $S$ is called \emph{Artin supersingular} if its formal Brauer group $\widehat{\Br}(S)$ is the formal additive group $\widehat\GG_{a}$ (see \cite{Ar74}, \cite{AM77} or \S\,\ref{subsect:SSandBr}), or equivalently, the Newton polygon associated to the second crystalline cohomology $H^{2}_{\cris}(S/W(k))$ is a straight line (of slope 1). On the other hand, Shioda \cite{Sh74} has introduced another notion of supersingularity for K3 surfaces by considering the algebraicity of the $\ell$-adic cohomology classes of degree 2, for any $\ell\neq p$: we say that $S$ is \emph{Shioda supersingular} if the first Chern class map
\begin{equation*}
c_1: \Pic(S)\otimes \QQ_\ell\rightarrow H^2_{\et}(S,\QQ_\ell(1))
\end{equation*}
is surjective; this condition is independent of $\ell$, as it is equivalent to the maximality of the Picard rank, \emph{i.e.}~$\rho_{S}=b_{2}(S)=22$. It is easy to see that Shioda supersingularity implies Artin supersingularity. Conversely, the Tate conjecture \cite{Tate65} for K3 surfaces over finite fields, solved in \cite{MR723215}, \cite{MR819555}, \cite{Ma14}, \cite{Ch13}, \cite{Pe15}, \cite{Ch16} and \cite{KM16}, implies that these two notions actually coincide for any algebraically closed fields of positive characteristic, \emph{cf.}~\cite[Theorem 4.8]{MR3524169}:
\begin{equation}\label{ArtinShioda}
\hbox{Shioda supersingularity $\Leftrightarrow $  Artin supersingularity. }
\end{equation}

Supersingularity being essentially a \emph{cohomological} notion, it is natural to look for its relation to \emph{geometric} properties. Unlike complex K3 surfaces, there exist \emph{unirational} K3 surfaces over fields of positive characteristic, first examples being constructed by Shioda in \cite{Sh74} and Rudakov--{\v S}afarevi{\v c} in \cite{RS78}; then Artin \cite{Ar74} and Shioda \cite{Sh74} observed that unirational K3 surfaces must have maximal Picard rank 22, hence are supersingular. Conversely, one expects that unirationality is a geometric characterization of supersingularity for K3 surfaces:

\begin{conjecture}[Artin \cite{Ar74}, Shioda \cite{Sh74}, Rudakov--{\v S}afarevi{\v c} \cite{RS78}]\label{con}
	A K3 surface is supersingular if and only if it is unirational. 
\end{conjecture}
This conjecture has been confirmed over fields of characteristic $2$ by Rudakov--{\v S}afarevi{\v c} \cite{RS78} via the existence of  quasi-elliptic fibration. See also \cite{Li15}, \cite{BL18} and \cite{BL19} for recent progress. 
Let us also record that by \cite{Fak02}, the Chow motive of a supersingular K3 surface is of Tate type and in particular that the Chow group of 0-cycles is isomorphic to $\ZZ$, thus contrasting drastically to the situation over the complex numbers, where $\CH_{0}$ is infinite dimensional by Mumford's celebrated observation in \cite{MR0249428}.

\subsection{Symplectic varieties}
We want to investigate the higher dimensional analogues of the above story for K3 surfaces. In the setting of complex geometry, the natural generalizations of K3 surfaces are the \emph{irreducible holomorphic symplectic manifolds} or equivalently, \emph{compact hyper-K\"ahler manifolds}. Those are by definition the simply connected compact K\"ahler manifolds admitting a Ricci flat metric such that the holonomy group being the compact symplectic group. Together with abelian varieties (or more generally complex tori) and Calabi--Yau varieties, they form the fundamental building blocks for compact K\"ahler manifolds with vanishing first Chern class, thanks to the Beauville--Bogomolov decomposition theorem \cite[Th\'eor\`eme 2]{MR730926}.   Holomorphic symplectic varieties over the complex numbers are extensively studied from various points of view in the last decades. We refer to the huge existing literature for more details \cite{MR730926}, \cite{MR1664696}, \cite[Part III]{MR1963559}.

In positive characteristic, there seems to be no accepted definition of \emph{symplectic varieties} (see however \cite{Ch16} and \cite{YangISV}). In this paper, we define them as 
smooth projective varieties $X$ defined over $k$ with trivial \'etale fundamental group and such that there exists a symplectic (\emph{i.e.}~nowhere-degenerate and closed) algebraic 2-form (Definition \ref{def:ISV}). When $k=\mathbb{C}$, our definition corresponds to products of irreducible holomorphic symplectic varieties. 

The objective of this paper is to initiate a systematic study of \emph{supersingular} symplectic varieties: we will discuss several natural definitions for the notion of supersingularity, propose some general conjectures and provide ample evidence for them. 

\subsubsection{Notions of supersingularity}
The notion(s) of supersingularity, which is subtle for higher-dimensional varieties,  can be approached in essentially two ways (see \S\,\ref{sect:SSgeneral}): via formal groups and $F$-crystal structures on the crystalline cohomology as Artin did \cite{Ar74}, or via the algebraicity of $\ell$-adic or crystalline cohomology groups as Shioda did \cite{Sh74}. More precisely, a symplectic variety $X$ is called 
\begin{itemize}
\item  \emph{$2^{nd}$-Artin supersingular}, if the $F$-crystal $H^{2}_{\cris}(X/W(k))$ is supersingular, that is, its Newton polygon is a straight line. Artin supersingularity turns out to be equivalent to the condition that the formal Brauer group $\widehat{\Br}(X)$ is isomorphic to $\widehat{\GG}_{a}$, provided that it is formally smooth of dimension 1 (Proposition \ref{prop:ArtinEquivBr}).
\item \emph{$2^{nd}$-Shioda supersingular}, if its Picard rank $\rho(X)$ is equal to its second Betti number $b_{2}(X)$. Or equivalently, the first Chern class map is surjective for the $\ell$-adic cohomology for any $\ell\neq p$.
\end{itemize}

The two notions are equivalent assuming the crystalline Tate conjecture for divisors (note that one can reduce to the finite field case by using the crystalline variational Tate conjecture for divisors proved by Morrow \cite{Mo14}). The reason of using the second cohomology in both approaches is the general belief that for a symplectic variety, its second cohomology group equipped with its enriched structure should capture a significant part of the information of the variety up to birational equivalence, as is partially justified by the Global Torelli Theorem over the field of complex numbers  \cite{VerbitskyTorelli, MarkmanSurvey, HuybrechtsTorelli}.  We will introduce nevertheless the notions of \textit{full Artin supersingularity}  and \textit{full Shioda supersingularity} which concern the entire cohomology ring of $X$. Roughly speaking, the former means that the Newton polygons of all crystalline cohomology groups are straight lines; the latter says that any cohomology group of even degree is spanned rationally by algebraic classes and any cohomology group of odd degree is isomorphic as $F$-crystal to the first cohomology of a supersingular abelian variety. See Definitions \ref{def:ArtinSS} and \ref{ShiodaSS} respectively.

\medskip
 Inspired by the results for K3 surfaces discussed before in \S\,\ref{subsection:SSK3}, we propose the following conjecture, generalizing the equivalence \eqref{ArtinShioda} as well as Conjecture \ref{con}. As will be explained later in \S\,\ref{sect:SSgeneral} and \S\,\ref{sec:SSISV}, all the listed conditions below are  \textit{a priori} stronger than the $2^{nd}$-Artin supersingularity.
 
\begin{named}{Conjecture A}[Characterizations of supersingularity]
Let $X$ be a symplectic variety defined over an algebraically closed field $k$ of positive characteristic. If it is $2^{nd}$-Artin supersingular, then all the following conditions hold:
	\begin{itemize}
		\item \textbf{\upshape (Equivalence conjecture)}  $X$ is fully Shioda (hence fully Artin) supersingular.
		\item \textbf{\upshape (Unirationality conjecture)}  $X$ is unirational. 
		\item \textbf{\upshape (RCC conjecture)} $X$ is rationally chain connected. 
		\item \textbf{\upshape (Supersingular abelian motive conjecture)}  The rational Chow motive of $X$ is a direct summand of the motive of a supersingular abelian variety.
		\end{itemize}
\end{named}
The Tate conjecture implies the equivalence between the $2^{nd}$-Artin supersingularity and $2^{nd}$-Shioda supersingularity; while the equivalence conjecture predicts more strongly that the supersingularity of the second cohomology implies the supersingularity of the entire cohomology.  The unirationality conjecture and the RCC conjecture can be viewed as a geometric characterization for the cohomological notion of supersingularity.   The supersingular abelian motive conjecture comes from the expectation that the algebraic cycles on a supersingular symplectic variety are ``as easy as possible''. 

By the work of Fakhruddin \cite{Fak02}, having supersingular abelian motive implies that the variety is fully Shioda supersingular (see Corollary \ref{cor:SSAV}). 
Therefore we can summarize Conjecture A as follows: the notions in the diagram of implications in Figure 1 below are all equivalent for symplectic varieties. Some of the implications in Figure 1 are not obvious and will be explained in \S\,\ref{sect:SSgeneral} and \S\,\ref{sec:SSISV}.

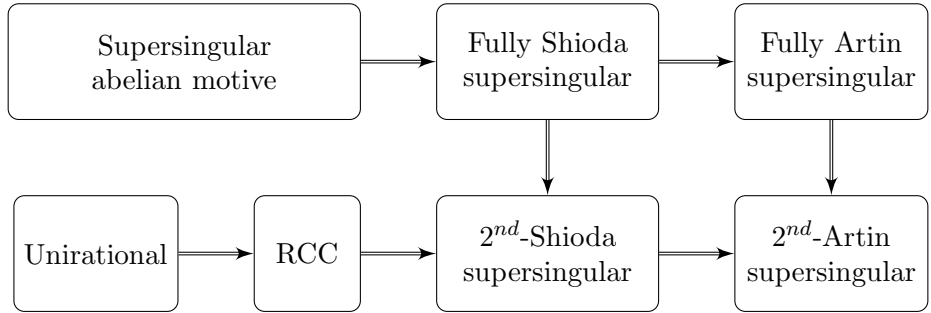
\begin{figure} 
\begin{tikzpicture}
\node[draw,rectangle,text centered, rounded corners,text width=11.5em,minimum height=4em] (AM){Supersingular abelian motive};
\node[draw,rectangle,text centered, rounded corners,text width=7em,minimum height=4em, right=1cm of AM] (FS){Fully Shioda supersingular};
\node[draw,rectangle,text centered, rounded corners,text width=7em,minimum height=4em, below=1cm of FS] (2S){$2^{nd}$-Shioda supersingular};
\node[draw,rectangle,text centered, rounded corners,text width=6em,minimum height=4em, right=1cm of FS] (FA){Fully Artin supersingular};
\node[draw,rectangle,text centered, rounded corners,text width=6em,minimum height=4em, below=1cm of FA] (2A){$2^{nd}$-Artin supersingular};
\node[draw,rectangle,text centered, rounded corners,text width=3em,minimum height=4em, left=1cm of 2S] (RCC){RCC};
\node[draw,rectangle,text centered, rounded corners,text width=5em,minimum height=4em, left=1cm of RCC] (UR){Unirational};
\draw[-latex',double](AM)--(FS);
\draw[-latex',double](FS)--(2S);
\draw[-latex',double](FS)--(FA);
\draw[-latex',double](FA)--(2A);
\draw[-latex',double](2S)--(2A);
\draw[-latex',double](UR)--(RCC);
\draw[-latex',double](RCC)--(2S);
\end{tikzpicture}
\caption{Characterizations of supersingularity for symplectic varieties}
\label{eqn:Implications}
\end{figure}

 \medskip
\subsubsection{Cycles and motives}
Carrying further the philosophy that the cycles and their intersection theory on a supersingular symplectic variety are as easy as possible, we propose the following conjecture, which is a fairly complete description of their Chow rings.  Recall that $\CH^{i}(X)_{\alg}$ is the subgroup of the Chow group $\CH^{i}(X)$ consisting of algebraically trivial cycles. Denote also by $\overline{\CH}^{i}(X)_{\QQ}$ the $i^{th}$ rational Chow group modulo numerical equivalence. There is a natural epimorphism $\CH^*(X)_\QQ\twoheadrightarrow \overline{
\CH}^*(X)_\QQ$.

\begin{named}{Conjecture B}[Supersingular Bloch--Beilinson--Beauville Conjecture]
Let $X$ be a supersingular symplectic variety defined over an algebraically closed field $k$. Then the following conditions hold:
\begin{enumerate}[$(i)$]
\item Numerical equivalence and algebraic equivalence coincide on $\CH^{*}(X)_{\QQ}$:  $$\CH^*(X)_{\alg, \QQ}=\ker(\CH^*(X)_\QQ\twoheadrightarrow \overline{\CH}^*(X)_\QQ).$$  In particular, the Griffiths groups of $X$ are of torsion.
\item  For any $0\leq i\leq \dim(X)$, there exists a regular surjective homomorphism $$\nu_{i}: \CH^{i}(X)_{\alg} \to Ab^{i}(X)(k)$$ to the group of $k$-points of a supersingular abelian variety $Ab^{i}(X)$ of dimension $\frac{1}{2}b_{2i-1}(X)$, called the \emph{algebraic representative}, such that $\nu_i$ has finite kernel and it is universal among regular homomorphisms from $\CH^{i}(X)_{\alg}$ to abelian varieties. 
\item There is a multiplicative decomposition
\begin{equation}\label{eqn:Decomposition}
\CH^{*}(X)_{\QQ}=\DCH^{*}(X)\oplus \CH^{*}(X)_{\alg, \QQ},
\end{equation}
where $\DCH^{*}(X)$, called the space of \emph{distinguished cycles}, is a graded $\QQ$-subalgebra of $\CH^*(X)_\QQ$ containing all the Chern classes of $X$. In particular, the natural map  $\DCH^*(X)\to \overline\CH^{*}(X)_{\QQ}$ is an isomorphism.

\item The intersection product restricted to the subring $\CH^{*}(X)_{\alg}$ is zero. In particular, it forms a square zero graded ideal.
\end{enumerate}
In other words, the $\QQ$-algebra $\CH^{*}(X)_{\QQ}$ is the square zero extension of a graded subalgebra isomorphic to $\overline\CH^{*}(X)_{\QQ}$ by a graded module $Ab^{*}(X)(k)_{\QQ}$. In particular, if all the odd Betti numbers vanish, then the natural map $\CH^*(X)_\QQ\to \overline{\CH^*}(X)_\QQ$ is an isomorphism and the rational Chow motive of $X$ is of Tate type.
\end{named}
As a consequence of Conjecture B, $\DCH^*(X)$ provides a $\QQ$-vector space such that for all $\ell\neq p$, the restriction of the cycle class map $\DCH^{*}(X)\otimes_\QQ {\QQ_{\ell}}\xrightarrow{\simeq} H^{2*}(X, \QQ_{\ell})$ is an isomorphism.
Moreover, the decomposition \eqref{eqn:Decomposition} is expected to be canonical and
the $\overline\CH^{*}(X)_{\QQ}$-module structure on $Ab^{*}(X)(k)_{\QQ}$ should be determined by, or at least closely related to, the $H^{2*}(X)$-module structure on $H^{2*-1}(X)$, where $H$ is any Weil cohomology theory.

\begin{remark}
The following remarks will be developed with more details later. 
\begin{itemize}
    \item Conjecture B can be seen as the supersingular analogue of Beauville's splitting property conjecture, which was formulated over the field of complex numbers in
    \cite{Beau07}. See \S\,\ref{subsect:Cycles} for the details.
    \item Conjecture B is implied by the combination of the equivalence conjecture, the Bloch--Beilinson conjecture (see Conjecture \ref{conj:SSBB+} for its fully supersingular version) and the section property conjecture \ref{conj:Section} proposed in Fu--Vial \cite{FuVial17}.
    \item We will also show in \S\,\ref{subsect:SSAV} that  the Bloch--Beilinson conjecture for supersingular symplectic varieties is a consequence of the supersingular abelian motive conjecture discussed above.
\end{itemize}
\end{remark}

To summarize:
\begin{figure}[H]
\begin{tikzpicture}
\node[draw,rectangle,text centered, rounded corners,text width=8em,minimum height=4em] (AM){Supersingular abelian motive conjecture};
\node[draw,rectangle,text centered, rounded corners,text width=12em,minimum height=4em, below=1cm of AM] (BB){Fully supersingular Bloch--Beilinson conjecture};
\node[draw,rectangle,text centered, rounded corners,text width=8em,minimum height=4em, left=1cm of BB](Eq){Equivalence conjecture};
\node[draw,rectangle,text centered, rounded corners,text width=8em,minimum height=4em, right=1cm of BB](Sec){Section property conjecture};
\node[draw,rectangle,text centered, rounded corners,text width=34em,minimum height=4em, below=1.2cm of BB](Sp){The ultimate description (Conjecture B):\\Supersingular Bloch--Beilinson--Beauville conjecture};
\draw[-latex',double](AM)--(BB);
\draw[-latex',double](AM)--(Eq);
\draw [decoration={brace, mirror, amplitude=15pt, raise=30pt},decorate] (Eq)--(Sec);
\end{tikzpicture}
\caption{Conjectures on algebraic cycles of supersingular symplectic varieties}
\label{eqn:Implications2}
\end{figure}
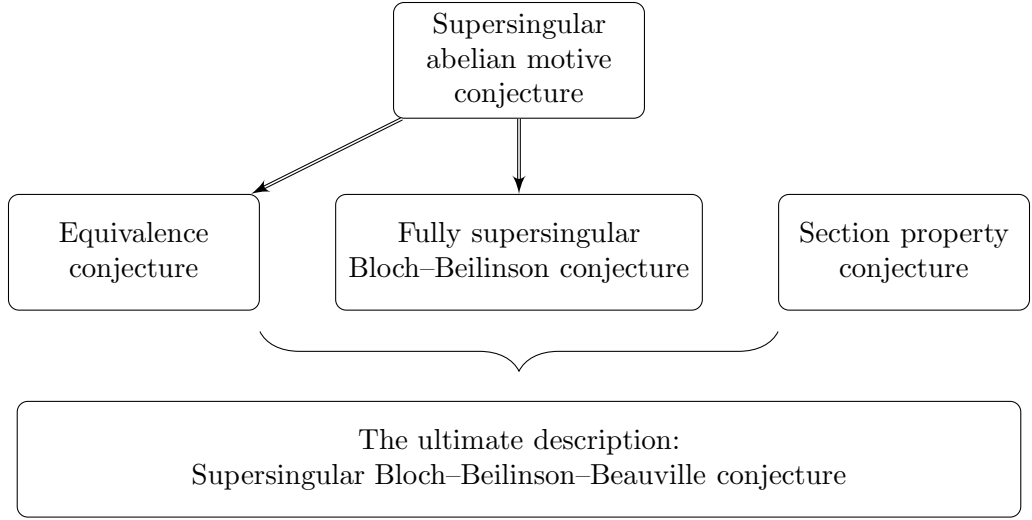

\subsection{Main results} Let $k$ be an algebraically closed field of characteristic $p>0$.
As in characteristic zero, two important families of examples of (simply connected) symplectic varieties are provided by the moduli spaces of stable sheaves on K3 surfaces and the Albanese fibers of the moduli spaces of stable sheaves on abelian surfaces. As evidence for the aforementioned conjectures, we establish most of them for most of the varieties in these two series. The key part is that we relate these moduli spaces birationally to punctual Hilbert schemes of surfaces, in the supersingular situation.

 Our first main result is the following one concerning the moduli spaces of sheaves on K3 surfaces.  
\begin{theorem}\label{mainthm1}
Let $S$ be a K3 surface defined over $k$. Let $H$ be an ample line bundle on $S$ and $X$ the moduli space of $H$-semistable sheaves on $S$ with Mukai vector $v=(r, c_{1}, s)$  satisfying $\left<v, v\right>\geq 0, r>0$ and $v$ is coprime to $p$. 
\begin{enumerate}[$(i)$]
	\item If $H$ is general with respect to $v$, then $X$ is a symplectic variety of dimension $2n=\left<v, v\right>+2$ and  deformation equivalent to the $n^{th}$ Hilbert scheme of points of $S$. Moreover, $X$ is  $2^{nd}$-Artin supersingular if and only if $S$ is supersingular. 
	\item If $X$ is $2^{nd}$-Artin supersingular,  then $X$ is irreducible symplectic and it is birational to the Hilbert scheme $S^{[n]}$, where $n=\frac{\left<v, v\right>+2}{2}$.  In particular, the unirationality conjecture holds for $X$ provided that $S$ is unirational. 
\item    If $X$ is $2^{nd}$-Artin supersingular,  the rational Chow motive of $X$ is of Tate type and the natural epimorphism $$\CH^{*}(X)_{\QQ}\twoheadrightarrow \overline{\CH^{*}}(X)_{\QQ}$$ is an isomorphism.  
In particular, the supersingular Bloch--Beilinson--Beauville conjecture holds for $X$ and the cycle class maps induce isomorphisms  $\CH^{*}(X)_{K}\simeq H^{*}_{\cris}(X/W)_{K}$ and $\CH^{*}(X)_{\QQ_{\ell}}\simeq H^{*}_{\et}(X, \QQ_{\ell})$ for all $\ell\neq p$.

\end{enumerate}  
\end{theorem}
Here, the notion of generality for $H$ and that of $v$ being coprime to $p$ are explained in Definitions \ref{def:vGeneric} and \ref{def:star} respectively. In fact, these conditions are the natural ones to ensure the smoothness of $X$. The coprime condition is automatically satisfied if $p\nmid(\frac{1}{2}\dim X-1)$, see Remark \ref{rmk:Exception}. Moreover, our results also hold for moduli spaces of twisted sheaves on supersingular K3 surfaces with Artin invariant at most $9$. See Theorem \ref{twistbir} and Corollary \ref{cor:TwistedK3} in \S\,\ref{section5}. \\

The second main result  is for moduli spaces of sheaves on abelian varieties. 
\begin{theorem}\label{mainthm2}
Let $A$ be an abelian surface defined over $k$. Let $H$ be an ample line bundle on $A$ and $X$ the Albanese fiber of the projective moduli space of $H$-semistable sheaves on $A$ with Mukai vector $v=(r, c_{1}, s)$ satisfying $\left<v, v\right>\geq 2$ and $r>0$.  
Denote $2n:=\left<v, v\right>-2$.
\begin{enumerate}[$(i)$]
	\item If $p\nmid (n+1)$ and if $H$ is general with respect to $v$, then $X$ is a smooth projective symplectic variety of dimension $2n$ and deformation equivalent to the $n^{th}$ generalized Kummer variety. Moreover, $X$ is  $2^{nd}$-Artin supersingular if and only if $A$ is supersingular. 
	\item Assume $A$ is supersingular and $p\nmid (n+1)$. Then $X$ is a $2^{nd}$-Artin supersingular irreducible symplectic variety and it is  birational to the $n^{th}$ generalized Kummer variety associated to some supersingular abelian surface $A'$.
	\item Under the assumption of $(ii)$, the supersingular abelian motive conjecture (hence the equivalence conjecture), the RCC conjecture and the supersingular Bloch--Beilinson--Beauville conjecture all hold for $X$.
	\end{enumerate}
\end{theorem}

In our subsequent work \cite{FLZ20-OG6}, the construction of O'Grady's 6-dimensional symplectic varieties \cite{OG6, PR13}, which are symplectic resolutions of certain singular moduli spaces of sheaves on abelian surfaces, is extended to positive characteristics, and Conjecture A and Conjecture B are proved for those varieties.\\

\noindent \textbf{Conventions: }  Throughout this paper, $k$ is an algebraically closed field of characteristic $p >0$,  $W=W(k)$ is its ring of Witt vectors, $W_i=W/{p^{i}W}$ is the $i$-th truncated Witt ring of $k$, and $K=K(W)=W[1/p]$ is the field of fractions of $W$. If $X$ is a variety defined over a field $F$ and  let $L$ be a field extension of $F$,  we write $X_L=X\otimes_{F} L$ for the base change. \\

\noindent \textbf{Acknowledgement:} The authors want to thank Nicolas Addington, Fran\c cois Charles, Cyril Demarche, Najmuddin Fakhruddin, Daniel Huybrechts, Qizheng Yin and Weizhe Zheng for very helpful discussions, and to thank the referees for their pertinent suggestions. The first author is also grateful to the Hausdorff Research Institute for Mathematics for the hospitality during the 2017 Trimester \emph{K-theory and related fields}.

\section{Generalities on the notion of supersingularity}\label{sect:SSgeneral}
In this section, we introduce several notions of supersingularity for cohomology groups of algebraic varieties in general and discuss the relations between them. 
\subsection{Supersingularity via $F$-crystals}\label{subsect:ArtinSS}

Let $X$ be a smooth projective variety  of dimension $n$ defined over  $k$.  The ring of Witt vectors $W=W(k)$ comes  equipped with a morphism $\sigma:W\rightarrow W$ induced by the Frobenius morphism $x\mapsto x^p$ of $k$.  For any $i\in \NN$, we denote by $H^i_{\cris}(X/W)$ the $i$-th integral crystalline cohomology group of $X$, which is a $W$-module whose rank is equal to the $i$-th Betti number\footnote{By definition, the $i$-th Betti number of $X$, denoted by $b_{i}(X)$, is the dimension (over $\QQ_{\ell}$) of the $\ell$-adic cohomology group $H^{i}_{\et}(X, \QQ_{\ell})$. The Betti number $b_{i}(X)$ is independent of the choice of the prime number $\ell$ different from $p$ (\cite{Weil1}, \cite{Weil2}, \cite{KatzMessing}).} of $X$. We set $$H^i (X):=H^i_{\cris}(X/W)/\text{torsion},~H^i(X)_K=H^i_{\cris}(X/W)\otimes_{W} K. $$  Then $H^i (X)$  is a free $W$-module and it is endowed  with a natural $\sigma$-linear map $$\varphi: H^i(X)\rightarrow H^i(X)$$ induced from the absolute Frobenius morphism $F:X\rightarrow X$ by functoriality. Moreover, by Poincar\'e duality, $\varphi$ is injective.

The pair $(H^i(X), \varphi)$ (\emph{resp.}~$(H^i(X)_K, \varphi_{K})$) forms therefore an $F$-crystal (\emph{resp.}~$F$-isocrystal), associated to which we have the \emph{Newton polygon} $\Nt^{i}(X)$ and the \emph{Hodge polygon} $\Hdg^{i}(X)$. According to Dieudonn\'e--Manin's classification theorem (\emph{cf.}~\cite{Ma63}), the $F$-crystal $(H^i(X), \varphi)$ is uniquely determined, up to isogeny, by (the slopes with multiplicities of) the Newton polygon $\Nt^i(X)$. Following \cite{Ma63}, we say that an $F$-crystal  $(M, \varphi)$ is {\it ordinary} if its Newton polygon and Hodge polygon agree and {\it supersingular} if its Newton polygon is a straight line.

\begin{definition}[Artin supersingularity]\label{def:ArtinSS}
For a given integer $i$, a smooth projective variety $X$ over $k$ is called $i^{th}$\emph{-Artin supersingular}, if the $F$-crystal $(H^{i}(X), \varphi)$ is supersingular. $X$ is called \emph{fully Artin supersingular} if it is $i^{th}$\emph{-Artin supersingular} for all integer $i$.
\end{definition}

\begin{example}[{Supersingular abelian varieties}]  \label{ex:SSAV}
Let $A$ be an abelian variety over $k$. We say that $A$ is {\it supersingular} if the $F$-crystal $(H^1(A),\varphi)$ is supersingular, \emph{i.e.}~$A$ is $1^{st}$-Artin supersingular. When $g=1$, this definition is equivalent to the classical notion of supersingularity for elliptic curves: for instance, the group of $p$-torsion points is trivial. 
We will show in \S\,\ref{subsect:SSAV} that a supersingular abelian variety is actually fully Artin supersingular.
\end{example}

\subsection{Supersingularity via cycle class map}\label{subsect:ShiodaSS}
We discuss the supersingularity in the sense of Shioda \cite{Sh74}. 
Let $X$ be a smooth projective variety defined over $k$. For any $r\in \NN$, there is the crystalline cycle class map 
\begin{equation}\label{cycle1}
\frc\frl^r\colon\CH^r(X)\otimes_{\ZZ} K \longrightarrow H^{2r}(X)_{K}:=H^{2r}_{\cris}(X/W)\otimes_{W} K,
\end{equation}
whose image lands in the eigenspace of eigenvalue $p^r$ with respect to the action of $\varphi$ on $H^{2r}(X)$.

\begin{definition}[Shioda supersingularity]\label{ShiodaSS}
	A smooth projective variety $X$ defined over $k$ is  called
	\begin{itemize}
	\item  \emph{${(2r)}^{th}$-Shioda supersingular}, if \eqref{cycle1} is surjective;
	\item \emph{even Shioda supersingular}\footnote{This is called \emph{fully rigged} in \cite[Definition 5.4]{vK03}.} if \eqref{cycle1} is surjective for all $r$;
	\item \emph{${(2r+1)}^{th}$-Shioda supersingular}, if there exist a supersingular abelian variety $A$ and an algebraic correspondence $\Gamma\in \CH^{\dim X-r}(X\times A)$ such that the cohomological correspondence $\Gamma_{*}: H^{2r+1}_{\cris}(X/W)_{K} \to  H^{1}_{\cris}(A/W)_{K}$ is an isomorphism;
	\item \emph{odd Shioda supersingular}, if it is ${(2r+1)}^{th}$-Shioda supersingular for all $r$;
	\item \emph{fully Shioda supersingular}, if it is even and odd Shioda supersingular.
	\end{itemize}  
\end{definition}

\begin{remark}[``Shioda implies Artin'']\label{rmk:ShiodaImpliesArtin}
Each notion of Shioda supersingularity is stronger than the corresponding notion of Artin supersingularity. More precisely,
\begin{enumerate}[$(i)$]
\item For even-degree cohomology, the $(2r)^{th}$-Shioda supersingularity  implies that the Frobenius action $\varphi$ is the multiplication by $p^{r}$ on the crystalline cohomology of degree $2r$, hence we have the $(2r)^{th}$-Artin supersingularity (\emph{i.e.}~the $F$-crystal $(H^{2r}(X),\varphi)$ is supersingular). The converse is implied by the combination of the crystalline Tate conjecture over finite fields (\emph{cf.}~\cite[7.3.3.2]{MR2115000}) and the crystalline variational Tate conjecture (see \cite[Conjecture 0.1]{Mo14}).
\item For odd-degree cohomology, the idea of Shioda supersingularity is ``of niveau 1'' and supersingular. The implication from $(2r+1)^{th}$-Shioda supersingularity to $(2r+1)^{th}$-Artin supersingularity follows from the definition and Example \ref{ex:SSAV}. The converse follows again from the crystalline Tate conjecture and its variational analogue.
\item Moreover,  one can try to define the (even) Shioda supersingularity using other Weil cohomology theory. For instance, for any prime number $\ell$ different from $p$, consider the $\ell$-adic cycle class map
	\begin{equation*}\label{cycle2}
	\frc\frl^r_\ell:\CH^r(X)\otimes \QQ_{\ell} \rightarrow H^{2r}_{\et}(X,\QQ_\ell(r)). 
	\end{equation*}
Note that the standard conjecture implies that the surjectivities of $\frc\frl^r_\ell$ and $\frc\frl^r$ are equivalent for all $\ell\neq p$, as the images of the cycle class maps for various Weil cohomology theories share the same $\QQ$-structure, namely $\overline{\CH}^{r}(X)_{\QQ}$, the rational Chow group modulo numerical equivalence.

When $r=1$, the $2^{nd}$-Shioda supersingularity does not depend on the choice of the Weil cohomology theory, as it is equivalent to say that the Picard rank is maximal.
\end{enumerate}
\end{remark}

\subsection{Full supersingularity and algebraic cycles}
We explore in this subsection the conjectural implications on the Chow groups of the full (Artin or Shioda) supersingularity discussed in \S\,\ref{subsect:ArtinSS} and \S\,\ref{subsect:ShiodaSS}. Recall that for a smooth projective variety, a cohomology group is called of (geometric) \textit{coniveau} at least $i$, if it is supported on a closed algebraic subset of codimension $i$.

Let $X$ be a smooth projective variety that is fully Shioda supersingular (Definition \ref{ShiodaSS}), which should be equivalent to being fully Artin supersingular assuming the (usual and variational) crystalline Tate conjecture \cite{Mo14}. Then by definition,
\begin{itemize}
\item All the even cohomology groups are of Tate type. In particular, $H^{2i}(X)$ is of coniveau $i$ for all $0\leq i\leq \dim X$.
\item All the odd cohomology groups are of abelian type. In particular, $H^{2i+1}(X)$ is of coniveau $\geq i$ for all $0\leq i\leq \dim X-1$.
\end{itemize} 
In particular, for any $2\leq j\leq i\leq \dim X$, the coniveau of $H^{2i-j}(X)$ is at least $i-j+1$.
By the general philosophy of coniveau, the conjectural Bloch--Beilinson filtration $F^{\cdot}$ on the rational Chow group satisfies that (\emph{cf.}~\cite[\S\,11.2]{MR2115000}, \cite[Conjecture 23.21]{MR1988456}) $$\operatorname{Gr}^{j}_{F}\CH^{i}(X)_{\QQ}=0, \forall j\geq 2, \forall i.$$ Therefore by the conjectural separatedness of the filtration $F^{\cdot}$, one expects that $F^{2}\CH^{i}(X)_{\QQ}=0$ in this case. It is generally believed that $F^{2}$ is closely related, if not equal, to the Abel--Jacobi kernel. Hence one can naturally conjecture that for fully supersingular varieties, all rational Chow groups are representable (\cite{MR0249428}), or better, their homologically trivial parts are represented by abelian varieties. The precise statement is Conjecture \ref{conj:SSBB+} below. Before that, let us explain the notion of representability with some more details. 

Recall first the notion of \emph{algebraic representatives} (\cite{MR0472843}, \cite{Murre85}) for Chow groups, developed by Murre. Let $\CH^{*}(X)_{\alg}$ be the group of algebraically trivial cycles modulo rational equivalence. 
\begin{definition}[Regular homomorphism {\cite[Definition 1.6.1]{Murre85}}]\label{def:Regular}
Let $X$ be a smooth projective variety and $A$ an abelian variety. 
A homomorphism $\phi: \CH^{i}(X)_{\alg}\to A(k)$ is called \emph{regular}, if for any family of algebraic cycles, that is, a connected pointed variety $(T, t_{0})$ together with a cycle $Z\in \CH^{i}(X\times T)$, the map
 \begin{eqnarray*}
T&\to& A(k)\\
t&\mapsto& \phi(Z_{t}-Z_{t_{0}})
\end{eqnarray*}
gives rise to a morphism of algebraic varieties $T\to A$.
\end{definition}

\begin{definition}[Algebraic representative \cite{Murre85}]\label{def:AlgRep}
Let $X$ be a smooth projective variety and $i$ a positive integer. An \emph{algebraic representative} for cycles of codimension $i$, is a couple $(\nu_{i}, Ab^{i})$, where $Ab^{i}$ is an abelian variety and $\nu_{i}: \CH^{i}(X)_{\alg}\to Ab^{i}(k)$ is a regular homomorphism (Definition \ref{def:Regular}), and it is \emph{universal} in the following sense: for any regular homomorphism to the group of $k$-points of an abelian variety $\phi: \CH^{i}(X)_{\alg}\to A(k)$, there exists a unique homomorphism of abelian varieties $\bar\phi: Ab^{i}\to A$ such that $\phi=\bar\phi\circ \nu_{i}$. It is easy to see that an algebraic representative, if exists, is unique up to unique isomorphism and $\nu_{i}$ is surjective (\cite[\S\,1.8]{Murre85}).
\end{definition}

As examples, the algebraic representative for the Chow group of divisors (\emph{resp.}~0-cycles) is the Picard variety (\emph{resp.}~the Albanese variety). The main result of \cite{Murre85} is the existence of an algebraic representative for codimension-2 cycles and its relation to the algebraic part of the intermediate Jacobian. However, the understanding of the kernel of $\nu_{i}$, when it is not zero, seems out of reach. 

In the spirit of the Bloch--Beilinson Conjecture, we can make the following speculation on the algebraic cycles on fully supersingular varieties.
\begin{conjecture}[Fully supersingular Bloch--Beilinson Conjecture]\label{conj:SSBB+}
Let $X$ be a smooth projective variety over $k$ which is fully Shioda supersingular. Then for any $0\leq i\leq \dim(X)$,
\begin{itemize}
\item Numerical equivalence and algebraic equivalence coincide on $\CH^{i}(X)_{\QQ}$. In particular, the rational Griffiths group $\Griff^{i}(X)_{\QQ}=0$.
\item  There exists a regular surjective homomorphism $$\nu_{i}: \CH^{i}(X)_{\alg} \to Ab^{i}(X)(k)$$ to the group of $k$-points on an abelian variety $Ab^{i}(X)$, which is the algebraic representative in the sense of Murre.
\item The kernel of $\nu_{i}$ is finite and $\dim Ab^{i}(X)=\frac{1}{2}b_{2i-1}(X)$.
\item $Ab^{i}$ is a supersingular abelian variety.
\item The intersection product restricted to $\CH^{*}(X)_{\alg}$ is zero.
\end{itemize}
In particular, the kernel of the algebra epimorphism $\CH^{*}(X)_{\QQ}\twoheadrightarrow \overline\CH^{*}(X)_{\QQ}$ is a square zero graded ideal given by the group of $k$-points on supersingular abelian varieties $Ab^{*}(X)(k)_{\QQ}:=\bigoplus_{i}Ab^{i}(X)(k)\otimes_{\ZZ}{\QQ}$.
\end{conjecture}

Thanks to the work of Fakhruddin \cite{Fak02}, Conjecture \ref{conj:SSBB+} is known for supersingular abelian varieties as well as other varieties with supersingular abelian motives. We will give an account of this aspect in \S\,\ref{subsect:SSAV}.

\subsection{Supersingular abelian varieties and their motives}\label{subsect:SSAV}
In this subsection, we illustrate the previous discussions in the special case of abelian varieties and establish some results on their motives for later use. 

Let $A$ be a $g$-dimensional abelian variety defined over $k$. 
For any $n\in \NN$, $H^n(A):=H^n_{\cris}(A/W)$ is a torsion free $W$-module of rank ${2g}\choose{n}$ and there exists a canonical isomorphism of $F$-crystals induced by the cup-product
\begin{equation}\label{Acrystal}
H^n(A)\cong \bigwedge^n H^1(A).
\end{equation}
By definition, $A$ is ($1^{st}$-Artin) supersingular if the $F$-crystal $H^{1}(A)$ is supersingular.
In this case, by \eqref{Acrystal}, all slopes of the $F$-crystal $(H^n(A),\varphi)$ are the same ($=\frac{n}{2}$), and hence $H^n(A)$ is as well a supersingular $F$-crystal for all $n\in \NN$, that is, $A$ is fully Artin supersingular (Definition \ref{def:ArtinSS}).

Before moving on to deeper results on cycles and motives of supersingular abelian varieties, let us recall some basics on their motivic decomposition. Denote by $\CHM(k)_{\QQ}$ the category of rational Chow motives over $k$. Building upon earlier works of Beauville \cite{MR726428} and \cite{MR826463} on Fourier transforms of algebraic cycles of abelian varieties, Deninger--Murre \cite{DM} produced a canonical motivic decomposition for any abelian variety (actually more generally for any abelian scheme; see \cite{MR1265530} for explicit formulae of the projectors):
$$\h(A)=\bigoplus_{i=0}^{2g}\h^{i}(A) \text{ in } \CHM(k)_{\QQ},$$
such that the Beauville component $\CH^{i}(A)_{(s)}$ is identified with $\CH^{i}(\h^{2i-s}(A))$. Furthermore, K\"unnemann \cite{MR1265530} showed that $\h^{i}(A)=\bigwedge^{i}\h^{1}(A)$ for all $i$, $\h^{2g}(A)\simeq \1(-g)$ and $\bigwedge^{i}\h^{1}(A)=0$ for $i>2g$ (see also Kings \cite{MR1609325}).

\begin{lemma}\label{lemma:SSEC}
Let $E$ be a supersingular elliptic curve over $k$. Then for any natural number $j$, we have in $\CHM(k)_{\QQ}$,
\begin{eqnarray*}
\Sym^{2j}\h^{1}(E)&\simeq &\1(-j)^{\oplus j(2j+1)}\\
\Sym^{2j+1}\h^{1}(E)&\simeq &\h^{1}(E)(-j)^{\oplus (j+1)(2j+1)}.
\end{eqnarray*}
\end{lemma}
\begin{proof}
As $\dim\End(E)_{\QQ}=4$, we have $\h^{1}(E)\otimes \h^{1}(E)\simeq \1(-1)^{\oplus 4}$. Hence as a direct summand, $\Sym^{2}\h^{1}(E)$ is also of Tate type, actually isomorphic to $\1(-1)^{\oplus 3}$. Since $\Sym^{2j}\h^{1}(E)$ is a direct summand of $\left(\Sym^{2}\h^{1}(E)\right)^{\otimes j}=\1(-j)^{\oplus 3^{j}}$, it is also of the form $\1(-j)^{\oplus m}$. The rank can be obtained by looking at the realization. As for the odd symmetric powers, $\Sym^{2j+1}\h^{1}(E)$ is a direct summand of $\Sym^{2j}\h^{1}(E)\otimes \h^{1}(E)=\h^{1}(E)(-j)^{\oplus j(2j+1)}$, therefore $\Sym^{2j+1}\h^{1}(E)(j)$ is a direct summand of $\h^{1}(E^{j(2j+1)})$. This corresponds to (up to isogeny) a sub-abelian variety of the supersingular abelian variety $E^{j(2j+1)}$, which must be itself supersingular, hence isogenous to a power of $E$ by \cite[Theorem 4.2]{Oo74}. In other words, $\Sym^{2j+1}\h^{1}(E)(j)\simeq \h^{1}(E^{m})$, for some $m\in \NN$, which can be identified by looking at the realization.
\end{proof}

We can now compute the Chow motives of supersingular abelian varieties.
\begin{theorem}\label{thm:SSAV}
Let $A$ be a $g$-dimensional supersingular abelian variety defined over an algebraically closed field $k$ of positive characteristic $p$. Let $b_{i}={{2g}\choose{i}}$ be the $i$-th Betti number of $A$. Then in the category $\CHM(k)_{\QQ}$, we have for any $i$, 
\begin{eqnarray}
\label{eqn:hevenA} \h^{2i}(A)&\simeq& \1(-i)^{\oplus b_{2i}};\\
\label{eqn:hoddA}\h^{2i+1}(A)&\simeq& \h^{1}(E)(-i)^{\oplus \frac{1}{2}b_{2i+1}},
\end{eqnarray}
where $E$ is a/any supersingular elliptic curve.
In particular, $A$ is fully Shioda supersingular.
\end{theorem}
\begin{proof}
By Oort's result \cite[Theorem 4.2]{Oo74}, an abelian variety $A$ is supersingular if and only if it is isogenous to the self-product of a/any  supersingular elliptic curve $E$.  In $\CHM(k)_{\QQ}$, we denote $\Lambda:=\1^{\oplus g}$ and then $\h^{1}(A)\simeq \h^{1}(E^{g})=\h^{1}(E)\otimes \Lambda$, since isogenous abelian varieties have isomorphic rational Chow motives. Standard facts on tensor operations in idempotent-complete symmetric mono\"idal categories (\emph{cf.}~\cite[Lecture 6]{MR1153249}) yields that
\begin{equation}\label{eqn:hiA}
 \h^{i}(A)\simeq\bigwedge^{i}\left(\h^{1}(E)\otimes \Lambda\right)\simeq \bigoplus_{\lambda\dashv i}\mathbb S_{\lambda}\h^{1}(E)\otimes \mathbb S_{\lambda'}\Lambda,
\end{equation}
where $\lambda$ runs over all partitions of $i$, $\lambda'$ is the transpose of $\lambda$ and $\mathbb S_{\lambda}$ is the Schur functor associated to $\lambda$. $\mathbb S_{\lambda'}\Lambda$ being a direct sum of the unit motive $\1$, let us take a closer look at $\mathbb S_{\lambda}\h^{1}(E)$. Recall that (\emph{cf.}~\cite{MR2167204} for example\footnote{The convention in \cite{MR2167204} is the graded/super one, which is the reason why the symmetric product and exterior product are switched from ours when applied to an ``odd" object $\h^{1}$ in \emph{loc.\,cit}. We prefer to stick to the ungraded convention so the comparison to the corresponding facts from classical cohomology theory is more transparent.})
\begin{itemize}
\item  $\mathbb S_{\lambda}\h^{1}(E)=0$ if $\lambda$ has length at least $3$;
\item if $\lambda$ has length at most 2, say $\lambda=(a+b, a)$ with $a, b\geq 0$ and $2a+b=i$, then since $\bigwedge^{2}\h^{1}(E)\simeq \1(-1)$ is a $\otimes$-invertible object, we have 
$$\mathbb S_{(a+b, a)}\h^{1}(E)=\left(\bigwedge^{2}\h^{1}(E)\right)^{\otimes a}\otimes \Sym^{b}\h^{1}(E)=\1(-a)\otimes \Sym^{b}\h^{1}(E).$$
\end{itemize}
Combining this with \eqref{eqn:hiA} and using Lemma \ref{lemma:SSEC}, we see that $\h^{i}(A)$ is a direct sum of some copies of $\1(-\frac{i}{2})$ if $i$ is even and a direct sum of some copies of $\h^{1}(E)(-\frac{i-1}{2})$ if $i$ is odd. The numbers of copies needed are easily calculated by looking at their realizations.
\end{proof}

\begin{definition}[Supersingular abelian motives]\label{def:SSAbMot}
Let $\CHM(k)_{\QQ}$ be the category of rational Chow motives over $k$. 
Let $\mathcal{M}^{ssab}$ be the idempotent-complete symmetric mono\"idal subcategory of $\CHM(k)_{\QQ}$ generated by the motives of supersingular abelian varieties. A smooth projective variety $X$ is said to have \emph{supersingular abelian motive} if its rational Chow motive $\h(X)$ belongs to $\mathcal{M}^{ssab}$.
\end{definition}
\begin{remark}
$\mathcal{M}^{ssab}$ contains the Tate motives by definition. Thanks to Theorem \ref{thm:SSAV}, $\mathcal{M}^{ssab}$ is actually generated, as idempotent-complete tensor category, by the Tate motives together with $\h^{1}(E)$ for a/any supersingular elliptic curve $E$. It can be shown that any object in $\mathcal{M}^{ssab}$ is a direct summand of the motive of some supersingular abelian variety. Therefore, for a smooth projective variety $X$, the condition of having supersingular abelian motive is exactly Fakhruddin's notion of ``strong supersingularity'' in \cite{Fak02}.
\end{remark}

The following result confirms in particular the Fully Supersingular Bloch--Beilinson Conjecture \ref{conj:SSBB+} and the full Shioda supersingularity for varieties with supersingular abelian motives. The results $(ii) - (v)$ are due to Fakhruddin \cite{Fak02}; while our proof presented below is somehow different and emphasizes the more fundamental result Theorem \ref{thm:SSAV}; $(i)$ and $(vi)$ are new according to authors' knowledge.

\begin{corollary}[\emph{cf.}~Fakhruddin \cite{Fak02}] \label{cor:SSAV}
Let $X$ be an $n$-dimensional smooth projective variety defined over an algebraically closed field $k$ of characteristic $p>0$, such that $X$ has supersingular abelian motive (Definition \ref{def:SSAbMot}). Let $b_{i}$ be the $i$-th Betti number of $X$. Then we have the following:
\begin{enumerate}[$(i)$]
\item In the category $\CHM(k)_{\QQ}$, we have 
\begin{equation}\label{eqn:hX}
 \h(X)\simeq \bigoplus_{i=0}^{n}\1(-i)^{\oplus b_{2i}}\oplus \bigoplus_{i=0}^{n-1}\h^{1}(E)(-i)^{\oplus \frac{1}{2}b_{2i+1}},
\end{equation}  
where $E$ is a supersingular elliptic curve.
\item $X$ is fully Shioda supersingular (Definition \ref{ShiodaSS}).
\item Numerical equivalence and algebraic equivalence coincide. In particular, for any integer $i$, the Griffiths group is of torsion: $\Griff^{i}(X)_{\QQ}=0$.
\item $\CH^{i}(X)_{\QQ}=\CH^{i}(X)_{(0)}\oplus \CH^{i}(X)_{(1)}$ with  $\CH^{i}(X)_{(0)}\simeq \QQ^{\oplus b_{2i}}$ providing a $\QQ$-structure for cohomology and $\CH^{i}(X)_{(1)}\simeq E(k)^{\frac{1}{2}b_{2i-1}}\otimes_{\ZZ}{\QQ}$ is the algebraically trivial part.
\item $\CH^{i}(X)_{\alg}$ has an algebraic representative $(\nu_{i}, Ab^{i})$ with $\ker(\nu_{i})$ finite and $Ab^{i}$ a supersingular abelian variety of dimension $\frac{1}{2}b_{2i-1}$.
\item The intersection product restricted to $\CH^{*}(X)_{\alg}$ is zero.
\end{enumerate}
\end{corollary}
\begin{proof}
For $(i)$, using Theorem \ref{thm:SSAV}, the motive of $X$ must be a direct sum of Tate motives and Tate twists of $\h^{1}(E)$. Then the precise numbers of Tate twists and copies are easily determined by looking at the realization.\\
$(ii)$ follows immediately from \eqref{eqn:hX} in $(i)$.\\
For $(iii)$, using $(i)$, it suffices to observe that numerical equivalence and algebraic equivalence are the same on elliptic curves and all the Griffiths groups of a point and an elliptic curve are trivial.\\
For $(iv)$, it is an immediate consequence of $(i)$ together with the simple fact that $\CH^{1}(\h^{1}(E))=\CH^{1}_{(1)}(E)\simeq \Pic^{0}(E)(k)\otimes_{\ZZ}\QQ\simeq E(k)\otimes_{\ZZ}\QQ$.\\
For $(v)$, we argue similarly as in \cite{Fak02}: denote $B=E^{\frac{1}{2}b_{2i-1}}$, then in \eqref{eqn:hX}, the only term contributes non-trivially to $\CH^{i}(X)_{\alg}$ is $\h^{1}(E)(1-i)^{\oplus b_{2i-1}}$. Now the isomorphisms \eqref{eqn:hX} of rational Chow motives can be interpreted as follows: there exist a positive integer $N$ and two correspondences $Z_{1}, Z_{2}\in \CH^{n-i+1}(X\times B)$, such that the two compositions of $$Z_{1}^{*}: \CH^{1}(B)_{\alg}=\Pic^{0}(B)(k)\to \CH^{i}(X)_{\alg}$$ and $$Z_{2,*} :\CH^{i}(X)_{\alg}\to \CH^{1}(B)_{\alg}=\Pic^{0}(B)(k)$$ are both the multiplication by $N$. By the divisibility of $\CH^{i}(X)_{\alg}$ and $\Pic^{0}(B)(k)$ (\emph{cf.}~\cite[Lemma 1.3]{MR2723320}), both $Z_{1}^{*}$ and $Z_{2,*}$ are surjective and the kernel of $Z_{2,*}$ is finite. The surjectivity of $Z_{2,*}$ implies the representability of $\CH^{i}(X)_{\alg}$ and in particular (by \cite{MR542191} for example) there exists an algebraic representative $\nu_{i}:\CH^{i}(X)_{\alg}\to Ab^{i}(k)$. By the universal property of algebraic representative, the (regular) surjective homomorphism $Z_{2,*}$ is factorized through $\nu_{i}$, hence the kernel of $\nu_{i}$ is finite and $\Pic^{0}(B)$ is dominated by $Ab^{i}$. On the other hand, the surjectivities of $Z_{1}^{*}$ and $\nu_{i}$ show that $Ab^{i}$ is also dominated by $\Pic^{0}(B)$. Therefore, $Ab^{i}$ is isogenous to $\Pic^{0}(B)$, hence is supersingular of dimension $\frac{1}{2}b_{2i-1}$.\\
Finally for $(vi)$, recall first that any algebraically trivial cycle on $X$ is a linear combination of cycles of the form $\Gamma_{*}(\alpha)$, where $\Gamma\in \CH(C\times X)$ is a correspondence from a connected smooth projective curve $C$ to $X$ and $\alpha$ is a 0-cycle of degree 0 on $C$. Therefore we only need to show that for any two connected smooth projective curves $C_{1}, C_{2}$, $\Gamma_{i}\in \CH(C_{i}\times X)$ and $\alpha_{i}\in \CH_{0}(C_{i})_{\deg 0}$ for $i=1, 2$, then $\Gamma_{1,*}(\alpha_{1})\cdot \Gamma_{2,*}(\alpha_{2})=0$ in $\CH(X)$. Indeed, let $\Gamma$ be the correspondence from $C_{1}\times C_{2}$ to $X$ given by the composition $\delta_{X}\circ(\Gamma_{1}\times \Gamma_{2})$, then $\Gamma_{1,*}(\alpha_{1})\cdot \Gamma_{2,*}(\alpha_{2})=\Gamma_{*}(\alpha_{1}\times \alpha_{2})$, here $\times$ is the exterior product. Now on one hand, the Albanese invariant of the cycle $\alpha_{1}\times \alpha_{2}\in \CH_{0}(C_{1}\times C_{2})_{\alg}$ is trivial. On the other hand, by the universal property of algebraic representative, whose existence is proved in $(v)$, we have the commutative diagram
\begin{equation*}
\xymatrix{
\CH_{0}(C_{1}\times C_{2})_{\alg}\ar[r]^{\alb=\nu_{2}} \ar[d]^{\Gamma_{*}}  & \Alb(C_{1}\times C_{2})(k) \ar[d]\\
\CH^{*}(X)_{\alg} \ar[r]^{\nu_{*}} & Ab^{*}(X)(k).
}
\end{equation*}
Hence $\Gamma_{*}(\alpha_{1}\times \alpha_{2})$ belongs to $ \ker(\nu_{*})$, which is a finite abelian group by $(v)$. In other words, the image of the intersection product $\CH^{*}(X)_{\alg}\otimes \CH^{*}(X)_{\alg}\to \CH^{*}(X)_{\alg}$ is annihilated by some integer. However, this image is also divisible, hence must be zero. 
\end{proof}

\begin{remark}[Beauville's conjectures on supersingular abelian varieties]
Let $A$ be an abelian variety of dimension $g$. Recall Beauville's decomposition \cite{MR826463} on rational Chow groups: for any $0\leq i\leq g$, $$\CH^{i}(A)_{\QQ}=\bigoplus_{s=i-g}^{i}\CH^{i}(A)_{(s)},$$ where $\CH^{i}(A)_{(s)}$ is the common eigenspace of the multiplication-by-$m$ map (with eigenvalue $m^{2i-s}$) for all $m\in \ZZ$. As is pointed out in \cite{Fak02}, when $A$ is supersingular, Corollary \ref{cor:SSAV} confirms all the conjectures proposed in \cite{MR726428} and \cite{MR826463},
and the corresponding Beauville decomposition takes the following form:
\begin{equation}
	\label{eqn:DecompSSAV}
	\CH^{*}(A)_{\QQ}=\CH^{*}(A)_{(0)}\oplus \CH^{*}(A)_{(1)},
\end{equation}
with $\CH^*(A)_{(0)}$ injects into the cohomology and $\CH^{*}(A)_{(1)}=\CH^*(A)_{\alg, \QQ}$. 
A remarkable feature of the Chow rings of supersingular abelian varieties that is not shared by other varieties with supersingular abelian motives in general is that the decomposition \eqref{eqn:DecompSSAV} is \emph{multiplicative}, that is, in addition to the properties $(i) - (vi)$ in Corollary \ref{cor:SSAV}, we have that the subspace $\CH^{*}(X)_{(0)}$ is closed under the intersection product. We will see in \S\,\ref{subsect:Cycles} how this supplementary feature is extended to another class of supersingular varieties, namely the supersingular symplectic varieties.
\end{remark}

\begin{remark}
There is more precise information on the Griffiths group of codimension-2 cycles on a supersingular abelian variety if the base field is the algebraic closure of a finite field of characteristic $p$: in \cite{MR1904085}, Gordon--Joshi showed that it is at most a $p$-primary torsion group.
\end{remark}

\section{Supersingular symplectic varieties}\label{sec:SSISV}
We now start to investigate the various notions of supersingularity introduced in the previous section, as well as their relations, for a special class of varieties, namely the \emph{symplectic varieties}.  As  in the case of K3 surfaces, we expect that all these notions are equivalent in this case.

\subsection{Symplectic and irreducible symplectic varieties}

\begin{definition}\label{def:ISV}
	Let $X$ be a connected smooth projective variety defined over $k$ of characteristic $p>0$, and  let $\Omega_{X/k}^2$ be the locally free sheaf of algebraic $2$-forms over $k$. $X$ is called {\it  symplectic}  if 
	\begin{enumerate}
		\item $\pi_1^{\et}(X)=0$;
		\item $X$ admits a nowhere degenerate \emph{closed} algebraic 2-form;
	\end{enumerate}
	In particular, $X$ is even-dimensional with trivial canonical bundle. When $\dim H^0(X,\Omega^2_{X/k})=1$, we say that $X$ is \textit{irreducible symplectic}. 
\end{definition}
\begin{remark}
	There is also a definition of $K3^{[n]}$-type irreducible symplectic varieties in the recent work of Yang \cite{YangISV}, which requires the liftability to characteristic $0$. In that case, the comparison of two definitions is not clear to us. 
\end{remark}

 The construction methods for  symplectic varieties over fields of positive characteristic are somehow limited as is in the case over the complex numbers.  Let us collect some examples. 
 \begin{example}\label{ex:HK}
\begin{enumerate}[$(i)$]
\item When $p>2$, two-dimensional symplectic varieties are nothing but K3 surfaces. When $p=2$, there is an extra class, namely, the supersingular Enriques surfaces, which are also irreducible symplectic but with $\dim H^1(\mathcal{O})=1$;
\item the Hilbert scheme of length-$n$ subschemes on a K3 surface (\cite{MR730926});
\item smooth moduli spaces of stable sheaves (or more generally stable complexes with respect to some Bridgeland stability condition) on a K3 surface (\cite{Mukai84}, \cite{O97}, \cite{MR1664696}, \cite{BM14}, \cite{BMProj14}), under some mild numeric conditions on the Mukai vector and the polarization (see Proposition \ref{msheaf});
\item the generalized Kummer varieties $K_n(A)$ associated to an abelian surface $A$ (\cite{MR730926}), provided that it is smooth over $k$ (the smoothness condition does not always hold, see \cite{Sc09}). By definition, it is the fiber of the isotrivial fibration $s: A^{[n+1]}\to A$, where $s$ sends a subscheme to the summation of its support (with multiplicities);
\item the Albanese fiber of a smooth moduli space of stable sheaves (or more generally Bridgeland stable objects in the derived category) on an abelian surface (\cite{Mukai84}, \cite{Yo01}), under some mild numerical conditions on the Mukai vector and the polarization, provided that it is smooth over $k$ (see Proposition \ref{msheaf:ab});
\item the Fano varieties of lines of smooth cubic fourfolds (\cite{BD85});
\item O'Grady provided in \cite{OG10} and \cite{OG6} other examples by symplectic resolutions of some singular moduli space of semistable sheaves on K3 and abelian surfaces. His construction was over $\CC$; see \cite{FLZ20-OG6} for the adaptation to characteristic $p>2$;
\item as is shown in the recent work \cite{Srivastava20Enriques}, the punctual Hilbert schemes of supersingular Enriques surfaces (which only exist in characteristic 2) are symplectic varieties but not irreducible symplectic. 
\end{enumerate}
\medskip
Some nice properties for irreducible symplectic varieties defined over the complex numbers, such as the Beauville--Bogomolov quadratic form on the second cohomology and Torelli theorems, can be expected to hold in positive characteristics. 
We refer the readers to \cite{LZ19} for such an attempt via the theory of \textit{displays}.

 \end{example}

\subsection{Artin supersingularity and formal Brauer group}\label{subsect:SSandBr}
As indicated by the global Torelli theorem over the complex numbers, we expect that the second cohomology of a symplectic variety should control most of its geometry, up to birational equivalence. This motivates us to single out the following most important piece in Definition \ref{def:ArtinSS}:
\begin{definition}\label{def:2ArtinSS}
Let $X$ be a symplectic variety of dimension $2n$ over $k$. $X$ is called $2^{nd}$\emph{-Artin supersingular} if the $F$-crystal $(H^2(X),\varphi)$ is supersingular, \emph{i.e.} the Newton polygon $\Nt^{2}(X)$ is a straight line (of slope 1).
\end{definition} 

For a K3 surface, Artin defined his supersingularity originally in \cite{Ar74} by looking at its \emph{formal Brauer group} $\widehat{\Br}$, which turns out to be equivalent to the supersingularity of the $F$-crystal $(H^2(X),\varphi)$ introduced above. 
More generally, Artin and Mazur \cite{AM77} observed that the formal Brauer group  actually fits into a whole series of formal groups.  Recall that if the functor\,
\begin{eqnarray*}
\widehat{\Br}_X:( \text{Artin local $k$-algebras}) &\rightarrow &(\text{Abelian groups})\\
R &\mapsto& \ker\left(H^2_{\rm \et}(X\times_k R, \mathbb{G}_{\mathrm{m}})\rightarrow H^2_{\rm \et}(X,\mathbb{G}_{\mathrm{m}})\right),
\end{eqnarray*}
is pro-representable by a formal group, then we call it the {\it formal Brauer group}  of $X$, denoted by $\widehat{\Br}(X)$.
Note that by \cite{AM77}, we have the pro-representability if $H^{1}(X, \cO_{X})=0$.

 Thus we  can make the following definition.
\begin{definition}\label{def:ArtinBrSS}
A symplectic variety $X$ is called {\it  Artin $\widehat{\Br}$-supersingular} if 
 $\widehat{\Br}_X $ is pro-represented by the formal additive group $\widehat{\GG}_{a}$.
\end{definition}

 One can obtain the following consequence which guarantees that the two notions of supersingularity in Definitions \ref{def:2ArtinSS} and \ref{def:ArtinBrSS} coincide under mild condition, which presumably always holds for symplectic varieties.

\begin{proposition}\label{prop:ArtinEquivBr}
Let $X$ be a symplectic variety.
If the functor $\widehat{\Br}_X$ is pro-representable, formally smooth and $H^{2}(X,\mathcal{O}_{X})\simeq k$, then $X$ is Artin $\widehat{\Br}$-supersingular if and only if the $F$-crystal $(H^2(X),\varphi)$ is supersingular, that is, $X$ is $2^{nd}$-Artin supersingular.
\end{proposition}
\begin{proof}
Under the hypothesis, $\widehat{\Br}(X)$ is a 1-dimensional formal Lie group as $\dim H^{2}(X,\mathcal{O}_{X})=1$, hence is classified by its height. It is the formal additive group if and only if the Dieudonn\'e--Cartier module $D(\Phi^{2}(X))\otimes {K}$ is zero.
Then we can conclude by \cite[Corollary 2.7]{AM77}, since  the Newton polygon of $(H^2(X),\varphi)$ is a straight line if and only if there is no sub-$F$-isocrystals with slopes strictly less than $1$. \end{proof}

Furthermore, similar to the case of K3 surfaces \cite[Theorem 1.1]{Ar74}, the Picard number behaves very well for families of Artin $\widehat{\Br}$-supersingular symplectic varieties:
\begin{corollary}\label{famss} 	
Let $\pi:\fX\rightarrow B$ be a smooth projective family of  Artin $\widehat{\Br}$-supersingular symplectic varieties over a connected base scheme $B$ over $k$. Write $\fX_b=\pi^{-1}(b)$ for $b\in B$. Assume either  $\Pic^\tau(\fX/ B)$ is smooth or  $H^2_{\cris}(\fX_b/W)$ are torsion free for any $b\in B$.
Then the Picard number $\rho(\fX_b)$ is constant for all $b\in B$. In particular,  all the fibers of $\pi$ are $2^{nd}$-Shioda supersingular if and only if one of them is  $2^{nd}$-Shioda supersingular.
\end{corollary}
 
 \begin{proof}
The proof is similar to \cite[Theorem 1.1]{Ar74}.  It suffices to show $\rho(\fX_\eta)=\rho(\fX_0)$ for any family over $B=\Spec k[[t]]$,  where $\fX_\eta$ is the generic fiber and $\fX_0$ is the special fiber.  The supersingular assumption indicates that the $F$-crystal $H^2(\fX_b/W)$ is constant in $b$. Then the result in \cite{Ar74, Mo14,deunp} implies that the cokernel of the specialization map
$$\NS(\fX_\eta)\rightarrow \NS(\fX_0)$$
is finite and annihilated by powers of $p$. It follows that $\rho(\fX_\eta)=\rho(\fX_0)$. 
 \end{proof}

\subsection{Artin supersingularity \emph{vs.} Shioda supersingularity }

As mentioned in Remark \ref{rmk:ShiodaImpliesArtin}, the $2^{nd}$-Artin supersingularity is \emph{a priori} weaker than the $2^{nd}$-Shioda supersingularity. In the other direction, we know the Tate conjecture holds for certain symplectic varieties.

 \begin{theorem}[Charles \cite{Ch13}]\label{Tatethm}
Let $Y$ be a symplectic variety of dimension $2n$  over $\CC$ and let $L$ be an ample line bundle on $Y$ and $d=c_1(L)^{2n}$.  
 Assume that $p$ is coprime to $d$ and that $p>2n$. Suppose that $Y$ can be defined over a finite unramified extension of $\QQ_p$ and that Y has good reduction at $p$. If the Beauville-Bogomolov form of $Y$ induces a non-degenerate quadratic form on the reduction modulo $p$ of the primitive lattice in the second cohomology group of $Y$,  then the reduction of $Y$ at $p$, denoted by $X$, satisfies the Tate conjecture for divisors. 
  \end{theorem}
 
This yields the following consequence:

\begin{corollary}
Suppose $X$ is a symplectic variety defined over $k$ satisfying all the conditions in Theorem \ref{Tatethm}, then $X$ is $2^{nd}$-Artin supersingular if and only if $X$ is $2^{nd}$-Shioda supersingular.  
\end{corollary}

 The more difficult question is to go beyond the second cohomology and ask whether $X$ is \emph{fully} Shioda supersingular (Definition \ref{ShiodaSS}) and hence \emph{fully} Artin supersingular (Definition \ref{def:ArtinSS}), if $X$ is $2^{nd}$-Artin supersingular; that is, whether the notions in the following diagram of implications are equivalent.

 \begin{equation*}
\xymatrix{
 \text{Fully Shioda supersingular}\ar@{=>}[d] \ar@{=>}[r]& \text{Fully Artin supersingular}\ar@{=>}[d]\\
 2^{nd}\text{-Shioda supersingular}\ar@{=>}[r]&  2^{nd}\text{-Artin supersingular}
}
\end{equation*}
This is the equivalence conjecture in the introduction (see Conjecture A).  Later we will verify this for smooth moduli spaces of semistable sheaves on K3 surfaces and abelian surfaces. A fundamental reason to believe this conjecture comes from a motivic consideration (the supersingular abelian motive conjecture in Conjecture A), which will be explained in \S\,\ref{subsect:Cycles}.



\subsection{Unirationality \emph{vs.} supersingularity}
In the direction of looking for a geometric characterization of supersingularity for symplectic varieties, we proposed in the introduction to generalize Conjecture \ref{con} for K3 surfaces to the unirationality conjecture and the RCC conjecture for symplectic varieties, which respectively says that the $2^{nd}$-Artin supersingularity is equivalent to the unirationality and the rational chain connectedness; see Conjecture A in the introduction.

  
One can get many examples of $2^{nd}$-Shioda supersingular varieties from varieties with ``sufficiently many'' rational curves.  The result below, which says that \emph{the rational chain connectedness implies the algebraicity of $H^{2}$}, is well-known in characteristic $0$, and it holds in positive characteristics  as well. 
  
\begin{theorem}[\emph{cf.}~{\cite[Theorem 1.2]{GJ17}}]\label{RCCss}
Let $X$ be smooth projective variety over $k$. If $X$ is rationally chain connected, then the first Chern class map induces an isomorphism  $\Pic(X)\otimes \QQ_\ell \cong H^2_{\et}(X, \QQ_\ell(1))$ for all $\ell\neq p $.   In particular, $\rho(X)=b_2(X)$ and $X$ is $2^{nd}$-Shioda supersingular.
 \end{theorem}
  
 

Theorem \ref{RCCss} gives the following implications, which is part of Figure \ref{eqn:Implications} in the introduction:
\begin{equation*}
\text{Unirational} \Longrightarrow \text{RCC}\Longrightarrow 2^{nd}\text{-Shioda supersingular}\Longrightarrow  2^{nd}\text{-Artin supersingular}.
\end{equation*} 

As evidence, we  will verify the unirationality conjecture for most smooth moduli spaces of semistable sheaves on K3 surfaces, assuming the unirationality of supersingular K3 surfaces (Theorem \ref{mainthm1}), and prove the RCC conjecture for most smooth Albanese fibers of moduli spaces of semistable sheaves on abelian surfaces (Theorem \ref{mainthm2}). However, the unirationality conjecture remains open even for generalized Kummer varieties of dimension at least four.  In a sequel to this paper \cite{FLZ20-OG6}, we investigate these conjectures for some symplectic varieties of O'Grady type.

\subsection{Algebraic cycles and supersingularity}\label{subsect:Cycles}
 Over the field of complex numbers, thanks to the Kuga--Satake construction (\emph{cf.}~\cite{KSV}), one expects that the Chow motive of any projective symplectic variety is \emph{of abelian type}. In positive characteristic, it is also natural to conjecture that any supersingular symplectic variety has \emph{supersingular abelian Chow motive} (Definition \ref{def:SSAbMot}), in the sense that its Chow motive belongs to the idempotent-complete additive tensor subcategory of $\CHM_{\QQ}$ generated by the motive of supersingular abelian varieties.
See the supersingular abelian motive conjecture in the introduction (Conjecture A) for the statement. 


Thanks to Theorem \ref{thm:SSAV} and Corollary \ref{cor:SSAV}, the supersingular abelian motive conjecture implies the equivalence conjecture (\emph{i.e.}~$2^{nd}$-Shioda supersingular implis fully Shioda supersingular) and the fully supersingular Bloch--Beilinson conjecture \ref{conj:SSBB+} for symplectic varieties: 
\begin{equation*}
\xymatrix{
&\text{Supersingular abelian motive}\ar@{=>}[d]\ar@{=>}[r]& \text{Fully supersingular Bloch--Beilinson}\\
& \text{Fully Shioda supersingular}&
}
\end{equation*}
which is part of the diagrams in Figure \ref{eqn:Implications} and Figure \ref{eqn:Implications2}.

One can summarize differently Corollary \ref{cor:SSAV}: for any fully supersingular variety having supersingular abelian motive, there is a short exact sequence of graded vector spaces
\begin{equation}\label{eqn:SES}
 0\to Ab^{*}(X)(k)_{\QQ}\to \CH^{*}(X)_{\QQ}\to \overline\CH^{*}(X)_{\QQ}\to 0,
\end{equation}
where $\overline\CH^{*}:=\CH^{*}/\equiv$ and $Ab^{*}(X)(k)_{\QQ}=\CH^{*}(X)_{\alg, \QQ}=\CH^{*}(X)_{\operatorname{num}, \QQ}$. However so far the only thing we can say in general about the ring structure on $\CH^{*}(X)_{\QQ}$ given by the intersection product is that $Ab^{*}(X)(k)_{\QQ}$ forms a square zero graded ideal. It is the insight of Beauville \cite{Beau07} that reveals a supplementary structure on $\CH^{*}(X)_{\QQ}$: if $X$ is moreover symplectic, then its rational Chow ring should have a multiplicative splitting of the Bloch--Beilinson filtration. In the fully supersingular case, the Bloch--Beilinson filtration is precisely \eqref{eqn:SES} and Beauville's splitting conjecture reduces to the supersingular case of the following conjecture.

\begin{conjecture}[Section property conjecture \cite{FuVial17}]\label{conj:Section}
Let $X$ be a symplectic variety over an algebraically closed field. Then the algebra epimorphism $\CH^{*}(X)_{\QQ}\twoheadrightarrow\overline\CH^{*}(X)_{\QQ}$ admits a (multiplicative) section  whose image contains all the Chern classes of the tangent bundle of $X$.
\end{conjecture}
The study of this conjecture in \cite{FuVial17} was inspired by O'Sullivan's theory on \emph{symmetrically distinguished} cycles on abelian varieties \cite{MR2795752}, which provides such an algebra section for the epimorphism $\CH^*(A)_\QQ\twoheadrightarrow \overline{\CH}^*(A)_\QQ$ for any abelian variety $A$. This is an important step towards Beauville's general conjecture in \cite{MR726428} and \cite{MR826463}. As a result, we call the elements in the image of the section in Conjecture \ref{conj:Section} \emph{distinguished cycles} and denote the corresponding subalgebra by $\DCH^{*}(X)$. By definition, the natural composed map $\DCH^*(X)\to \overline\CH^*(X)_\QQ$ is an isomorphism.

\begin{remark}
It is easy to deduce the supersingular Bloch--Beilinson--Beauville conjecture (=Conjecture B in the introduction) 
from the other aforementioned conjectures, thus completing Figure \ref{eqn:Implications2}. First of all, for a symplectic variety, the equivalence conjecture allows us to go from the $2^{nd}$-Artin supersingularity to the full Shioda supersingularity. Now items $(i), (ii), (iv)$ of Conjecture B are included in the fully supersingular Bloch--Beilinson conjecture \ref{conj:SSBB+}. The only remaining condition is in item $(iii)$ that $\DCH^{*}(X)$ is closed under the intersection product and contains Chern classes of $X$; this is exactly the content of the section property conjecture \ref{conj:Section}. 

To make the link to Beauville's original splitting property conjecture \cite{Beau07} more transparent, we record that in his notation,   $\CH^{*}(X)_{(0)}:=\DCH^{*}(X)$ and $\CH^{*}(X)_{(1)}=\CH^{*}(X)_{\alg, \QQ}$. The fully supersingular Bloch--Beilinson conjecture \ref{conj:SSBB+} gives that $\CH^{*}(X)_{(0)}\oplus \CH^{*}(X)_{(1)}=\CH^{*}(X)_{\QQ}$ and $\CH^{*}(X)_{(1)}$ is a square zero ideal; while the ``splitting" says simply that $\CH^*(X)_{(0)}$ is closed under the product.
\end{remark}

\subsection{Birational symplectic varieties}
In this subsection, we compare the Chow rings and cohomology rings of two birationally equivalent symplectic varieties. 

The starting point is the following result of Rie{\ss}, built on Huybrechts' fundamental work \cite{MR1664696}. Actually her proof yields the following more precise result. We denote by $\Delta_{X}\subset X\times X$ the diagonal and $\delta_{X}=\{(x,x,x) ~\mid~ x\in X\}\subset X\times X\times X$ the small diagonal for a variety $X$.

\begin{theorem}[{Rie{\ss}  \cite[\S\,3.3 and Lemma 4.4]{MR3268859}}]\label{thm:Riess}
Let $X$ and $Y$ be $d$-dimensional projective symplectic varieties over $k=\CC$. If they are birational, then there exists a correspondence $Z\in \CH_{d}(X\times Y)$ such that 
\begin{enumerate}[$(i)$]
			\item $(Z\times Z)_{*}: \CH_{d}(X\times X)\to \CH_{d}(Y\times Y)$ sends
			$\Delta_{X}$ to $\Delta_{Y}$;
			\item $(Z\times Z\times Z)_{*}: \CH_{d}(X\times X\times X)\to \CH_{d}(Y\times
			Y\times Y)$ sends $\delta_{X}$ to $\delta_{Y}$.
			\item $Z_{*}: \CH(X)\to \CH(Y)$ sends $c_i(X)$ to $c_{i}(Y)$ for any $i\in \NN$;
			\item $Z$ induces an isomorphism of algebra objects $\h(X)\to \h(Y)$ in $\CHM(k)$
			with inverse given by ${}^{t}Z$.
		\end{enumerate}
		In particular, $Z$ induces an isomorphism between their Chow rings and
		cohomology rings.
\end{theorem}
	
Note that $(iv)$ is a reformulation of $(i)$ and $(ii)$. Our objective is to extend Rie{\ss}'s result to other algebraically closed fields under the condition of liftability. More precisely,
	
\begin{proposition}\label{prop:BirHK}
Let $X$ and $Y$ be two birationally equivalent projective symplectic varieties defined over an algebraically closed field $k$. If the characteristic of $k$ is positive, we assume moreover that $X$ and $Y$ are both liftable to $\mathcal{X}$ and $\mathcal{Y}$ over a characteristic zero base $W$ with geometric generic fibers $\mathcal{X}_{\overline{\eta}_{W}}$ and $\mathcal{Y}_{\overline{\eta}_{W}}$ being birational symplectic varieties. Then the same results as in Theorem \ref{thm:Riess} hold.
\end{proposition}
\begin{proof}
We first treat the case where $\car(k)=0$. Without loss of generality, we can assume that $k$ is finitely generated over its prime field $\QQ$, and fix an embedding $k\hookrightarrow \CC$. As $X_{\CC}$ and $Y_{\CC}$ are birational complex symplectic varieties, we have a cycle $Z_{\CC}\in \CH^{d}(X_{\CC}\times_{\CC}Y_{\CC})$ verifying the properties in Theorem \ref{thm:Riess}. Let $L$ be a finitely generated field extension of $k$, such that $Z_{\CC}$, as well as all rational equivalences involved, is defined over $L$. Take a smooth connected $k$-variety $B$ whose function field $k(B)=\kappa(\eta_{B})=L$ and choose a closed point $b\in B(k)$. Let $Z\in \CH^{d}(X_{\eta_{B}}\times_{L}Y_{\eta_{B}})$ with $Z\otimes_{L}\CC=Z_{\CC}$ and such that $Z$ satisfies the properties in Theorem \ref{thm:Riess} for $X_{\eta_{B}}$ and $Y_{\eta_{B}}$. Now the specialization of $Z$ from the generic point $\eta_{B}$ to the closed point $b$ gives rise to a cycle $\spe(Z)\in \CH^{d}(X\times Y)$, which satisfies all the properties of Theorem \ref{thm:Riess} because specialization respects compositions of correspondences, diagonals and small diagonals (\emph{cf.}~\cite[\S\,20.3]{MR1644323}).
	
In the case where $\car(k)>0$, let $W$, $\mathcal{X}$ and $\mathcal{Y}$ be as in the statement. Denote by $K=\Frac(W)$, $w$ the closed point of $W$ with residual field $k$. By hypothesis, $\mathcal{X}_{K}$ and $\mathcal{Y}_{K}$ are geometrically birational, hence the result in characteristic zero proved in the previous paragraph implies that there exist a finite field extension $L/K$, an algebraic cycle $Z\in \CH^{d}(\mathcal{X}_{L}\times_{L} \mathcal{Y}_{L})$ which satisfies all the properties of Theorem \ref{thm:Riess} for $\mathcal{X}_{L}$ and $\mathcal{Y}_{L}$. Take any $W$-scheme $B$ with $\kappa(\eta_{B})=L$ and choose a closed point $b$ of $B$ in the fiber of $w$, then $\kappa(b)=k$ since $k$ is algebraically closed. Then as before, the specialization of $Z$ from the generic point $\eta_{B}$ to the closed point $b$ yields a cycle $\spe(Z)\in \CH^{d}(X\times_{k} Y)$ which inherits all the desired properties from $Z$.
\end{proof}

In view of Proposition \ref{prop:BirHK}, the following notions are convenient.
\begin{definition}[(Quasi-)liftably birational equivalence]\label{def:LiftBir}
Two symplectic varieties $X$ and $Y$ are called \emph{liftably birational} if they are both liftable to $\mathcal{X}$ and $\mathcal{Y}$ over some base $W$ of characteristic zero with geometric generic fibers being birationally equivalent symplectic varieties.

Two symplectic varieties $X$ and $Y$ are called \emph{quasi-liftably birational} if there exists a (finite) sequence of symplectic varieties $$X=X_{0},  X_1, \cdots,  X_m=Y,$$ such that for any $0\leq i\leq m-1$, $X_i$ and $X_{i+1}$ are liftably birational. 
\end{definition}
 
\begin{remark}
Two quasi-liftably birational symplectic varieties are indeed birationally equivalent. To see this, note that all symplectic varieties, having trivial canonical bundle, are non-ruled. Then it follows from \cite[Theorem~1]{MM64} that the birational equivalence between the geometric generic fibers of $\mathcal{X}$ and $\mathcal{Y}$ implies that $X$ and  $Y$ are birationally equivalent after possibly taking a field extension of the residue field $\kappa$. As $k$ is algebraically closed, $X$ and $Y$ are birationally equivalent over $k$ as well.  
\end{remark}

\begin{corollary}
Let $X$ and $Y$ be two quasi-liftably birational symplectic varieties. Then $X$ is $2^{nd}$-Artin (resp.~$2^{nd}$-Shioda, fully Artin, fully Shioda) supersingular if and only if $Y$ is so.
\end{corollary}
\begin{proof}
By Proposition \ref{prop:BirHK} $(iv)$, the rational Chow motives of $X$ and $Y$ are isomorphic. By realization, the crystalline cohomology groups of $X$ and $Y$ are isomorphic as $F$-isocrystals, hence we have the statements for $2^{nd}$ and full Artin supersingularities. Similarly, the algebraicity of cohomology is controlled by the motive, hence the statements on the $2^{nd}$ and full Shioda supersingularities follow. 
\end{proof}

\section{Moduli spaces of stable sheaves on K3 surfaces}\label{sect:ModuliK3}

\subsection{Preliminaries on K3 surfaces over positive characteristic} 
Some useful facts needed later on K3 surfaces are collected here. We start with the N\'eron--Severi lattices of supersingular K3 surfaces. Let  $S$ be a smooth projective K3 surface defined over $k$, an algebraically closed field of positive characteristic $p$. 	
As the Tate conjecture holds for K3 surfaces over finite fields of any characteristic (\cite{Ch13}, \cite{Pe15}, \cite{Ch16}, \cite{KM16}),  $S$ is Artin supersingular if and only if it is Shioda supersingular (\emph{cf.}~the argument in \cite[Theorem~4.8]{MR3524169}), hence the notion of supersingularity has no ambiguity for K3 surfaces.

By Artin's work \cite{Ar74}, if $S$ is supersingular, 
the discriminant of the
intersection form on $\NS(S)$ is 
$
{\rm disc}\,\NS(S)\,=\,\pm p^{2\sigma(S)}
$
for some integer $1\leq \sigma(S)\leq 10$, which is called the {\em Artin invariant} of $S$.   The lattice $\NS(S)$ is uniquely determined, up to isomorphism,  by its Artin invariant $\sigma(S)$. Such lattices are completely classified by Rudakov--{\v S}afarevi{\v c} in \cite[\S\,2]{RS78} and when $p>2$, there is a refinement due to Shimada \cite{Sh04}.

We summarize their results as follows. Given a lattice $\Lambda$ and an integer $n$, we denote by $\Lambda(n)$ the lattice obtained by multiplying its bilinear form by $n$. 

\begin{proposition}[Supersingular K3 lattices]\label{NS-ss}
Let $S$ be a supersingular K3 surface defined over an algebraically closed field of positive characteristic $p$. The lattice $\NS(S)$ is isomorphic to $-\Lambda_{\sigma(S)}$. When $p>2$, the intersection form of $\Lambda_{\sigma(S)}$ is given as below
\begin{enumerate}[$(i)$]
\item $\sigma(S)<10$,   $\Lambda_{\sigma(S)}=\begin{cases}
		U\oplus V^{(p)}_{20,2\sigma(S)}, &\text{if}~ p\equiv 3\mod 4, \text{ and } 2\nmid \sigma(S) \\ U\oplus H^{(p)}\oplus V^{(p)}_{16,2\sigma(S)} , &\text{otherwise}
		\end{cases}$
\item $\sigma(S)=10$, $\Lambda_{10}= U(p)\oplus H^{(p)}\oplus V^{(p)}_{16, 16}$,
\end{enumerate}
Here, $U$ is the hyperbolic plane, 
\begin{equation}
H^{(p)}=\left(\begin{array}{cccc}2& 1& 0&0  \\1 & (q+1)/2 & 0 &\gamma \\ 0& 0& p(q+1)/2 & p\\ 
	0& \gamma & p & 2(p+\gamma^2)/q
	\end{array}\right) 
	\end{equation}
$\text{satisfying that the prime}~q\equiv 3\mod 8, ~(\frac{-q}{p})=-1,~ \gamma^2+p\equiv 0\mod q $,
	and 	$$ V_{m,n}^{(p)}=V_0\cup (\frac{1}{2}\sum\limits_{1}^m e_i+V_0)$$  where
\begin{equation}\begin{aligned}
	V_0=\left <\sum\limits_{i=1}^m a_i e_i|~\sum a_i \equiv 0\mod 2\right>\subseteq \oplus \ZZ e_i;~ e_ie_j=\begin{cases}
	0, & {\rm if }~ i\neq j, \\ 1, & {\rm if}~ i=j>n,\\p, & {\rm otherwise}.
	\end{cases}
	\end{aligned}
\end{equation} 
\end{proposition}

When $p=2$, $\Lambda_{\sigma(S)}$ has been explicitly classified in \cite[\S\,2, P.157]{RS78}, we will only use the fact that it  contains the hyperbolic lattice $U$ as a direct summand when $\sigma(S)$ is odd and $U(2)$ when $\sigma(S)$ is even.  

As a consequence, Liedtke \cite[Proposition~3.9]{Li15} showed the existence of elliptic fibrations on supersingular K3 surfaces ($p>3$). As we need elliptic fibrations with more special properties on supersingular K3 surfaces, we will prove a strengthening of Liedtke's result later using Proposition~\ref{NS-ss}. 

Let us turn to the liftability problem of K3 surfaces. Ogus \cite[Corollary~2.3]{Og79} (attributed also to Deligne) showed that every polarized K3 surface  admits a projective lift. We will need the following stronger result.
\begin{proposition}[Lifting K3 surfaces with line bundles {\cite[Corollary 4.2]{LM11}}, {\cite[Appendix A.1 (iii)]{LO15}, \cite[Proposition 1.5]{Ch16}}]\label{lift}
Let $S$ be a smooth K3 surface over an algebraically closed field $k$ of characteristic $p>0$. Let $\Sigma\subseteq \NS(S)$ be a saturated subgroup of rank $<11$ containing an ample class. Then there exist a complete discrete valuation ring  $W'$ of characteristic $0$ with fraction field $K'$ and residue field $\kappa$ containing $k$, a relative K3 surface 
\begin{equation}
\cS \rightarrow \Spec(W')
\end{equation}
with special fiber $\cS_\kappa\cong S\times_k  \kappa $, such that the specialization map $\NS(\cS_{\overline{K'}})\to \NS(\cS_{\kappa})\simeq \NS(S)$ induces an isomorphism
\begin{equation}
 \NS(\cS_{\overline{K'}}) \xrightarrow{\simeq} \Sigma,
\end{equation}
where $\cS_{\overline{K'}}$ is the geometric generic fiber over the algebraic closure $\overline{K'}$. 
\end{proposition} 

\subsection{Moduli spaces of stable sheaves} One important source of examples of $K3^{[n]}$-type symplectic varieties is the moduli spaces of stable sheaves on a K3 surface (\cite{Mukai84}, \cite{O97}, \cite{MR1664696}). Given a K3 surface $S$, we denote by  $$\widetilde H(S)=\ZZ\cdot {\mathds 1}\oplus \NS(S)\oplus \ZZ\cdot\omega$$ the \emph{algebraic Mukai lattice} of $S$, where $\mathds 1$ is the fundamental class of $S$ and $\omega$ is the class of a point. An element $r\cdot {\mathds 1}+L+s\cdot \omega$ of $\widetilde H(S)$, with $r, s\in \ZZ$ and $L\in \NS(S)$, is often denoted by $(r, L, s)$. The lattice structure on $\widetilde H(S)$ is given by the following \emph{Mukai pairing} $\left<-,-\right>$: 
\begin{equation}\label{pairing}
\left<(r, L, s), (r', L', s')\right>=L\cdot L'-rs'-r's \in \ZZ.
\end{equation}
For a coherent sheaf $\cF$ on $S$, its \emph{Mukai vector} is defined by
$$v(\cF):=\ch(\cF)\sqrt{\td(S)}= \left(\rk(\cF), c_1(\cF), \chi(\cF)-\rk (\cF) \right)\in \widetilde{H}(S). $$

\begin{definition}[General polarizations]\label{def:vGeneric}
Let $v\in \widetilde{H}(S)$ be a primitive\footnote{that is, $v$ cannot be written as $mv'$ for an integer $m\geq 2$ and $v'$ another element in $\widetilde H(S)$.} Mukai vector. We say a polarization $H$ is {\it general with respect to $v$} if every $H$-semistable sheaf with Mukai vector $v$ is $H$-stable.  
\end{definition}


For instance, if $v=(r,c_1,s)$, one can easily show that  $H$ is general with respect to $v$ if it satisfies the  following numerical condition ({\it cf.}~\cite{Ch16}):
\begin{equation}\label{eqn:GCD3}
\gcd(r, c_1\cdot H , s )=1.
\end{equation}

Given a primitive element $v=(r,c_1,s)\in \widetilde H(S)$ such that $r>0$ and $\left<v,v\right>\geq 0$,  together with a general ample line bundle $H$ (with respect to $v$) on $S$,  we consider the moduli space of Gieseker--Maruyama $H$-stable sheaves on $S$ with Mukai vector $v$, denoted by $\cM_H(S,v)$.  According to the work of Langer \cite{La04},  $\cM_H(S,v)$ is a quasi-projective scheme over $k$. When $\cha (k)=0$, the works of Mukai \cite{Mukai84} ,  O'Grady \cite{O97}  and  Huybrechts \cite{MR1664696}  show that $\cM_H(S,v)$ is an irreducible symplectic variety of dimension $2n=\left<v,v\right>+2$ and  is of $K3^{[n]}$-deformation type.  Over fields of positive characteristic, the following analogous properties of $\cM_H(S,v)$ hold. 
\begin{proposition}[\emph{cf.}~{\cite[\S\,2]{Ch16}}]\label{msheaf}
Let $S$ be a smooth projective K3 surface, $H$ an ample line bundle and $v=(r, c_1,s)\in \widetilde{H}(S)$ be a primitive Mukai vector with $r>0$ and $\left<v, v\right>> 0$. If $H$ is general with respect to $v$, then 
\begin{enumerate}[(i)]
\item $M_H(S,v)$ is a smooth projective, symplectic variety of dimension $2n=\left<v, v\right>+2$ over $k$ and it is deformation equivalent to the $n$-th Hilbert scheme $S^{[n]}$.

\item  When $\ell \neq p$, there is a canonical quadratic form on $H^2(M_H(S,v), \ZZ_\ell(1))$. Let $v^\perp$ be the orthogonal complement of $v$ in the $\ell$-adic Mukai lattice of $S$.  There is an injective isometry
\begin{equation}\theta_v: v^\perp\cap \widetilde H(S)\rightarrow \NS(M_H(S,v)), 
		\end{equation}\label{NSiso}
		whose cokernel is a $p$-primary torsion group. Here $ v^\perp\cap \widetilde H(S)$ is the orthogonal complement of  $v$ in $\widetilde H(S)$.  
		\item There is an isomorphism of Galois representations (resp.~isocrystals)
		\begin{equation}
		v^\perp\otimes K \rightarrow H^2(M_H(S,v), K)
\end{equation}
between the rational \'etale (resp.~crystalline) cohomology groups, where $K=\QQ_\ell$ (resp.~the fraction field of $W$). Here $v^\perp\otimes K$ is the orthogonal complement of $v$ in $H^*(S,K)$.  It is compatible with the isometry \eqref{NSiso} via the \'etale (resp.~crystalline) cycle class map. 
\end{enumerate}
\end{proposition}

\begin{proof} All the assertions are well-known over  fields of  characteristic zero ({\it cf.}~\cite{Mukai84}, \cite{O97}). When $\cha(k)>0$ and assuming condition \eqref{eqn:GCD3}, the statements were  proved in \cite[Theorem~2.4]{Ch16}  via lifting to characteristic zero. But there is no difficulty to extend them under the more general assumption that $H$ is general with respect to $v$. Here we just sketch the proof of the missing parts in \cite{Ch16}: 

In part $(i)$, it remains to prove that $M_H(S,v)$ has trivial \'etale fundamental group. One can use Proposition \ref{lift} to  find a projective lift  
\begin{equation}\label{liftk31}
\cS\rightarrow \Spec W', 
\end{equation}
of $S$ over some discrete valuation ring $W'$ such that  both $H$ and $v$ lift to $\cS$, denoted by $\cH\in \Pic(\cS)$ and $v$, respectively. We consider the relative moduli space of  semistable sheaves on $\cS$ over $W'$ 
\begin{equation}\label{lift1}M_{\cH}(\cS,v)\rightarrow \Spec W'. 
\end{equation}
Let $\cS_\eta$ be the generic fiber of \eqref{liftk31} and let $\cH_\eta$ be the restriction of $\cH$ to $\cS_\eta$. Observe that the ample line bundle $\cH_\eta$ on the generic fiber $\cS_\eta$ remains general with respect to $v$ ({\it cf}.~\cite[Theorem~4.19 (1)]{BL18}), so the geometric generic fiber of \eqref{lift1} is a smooth projective symplectic variety. Now one can conclude by the general fact that the specialization morphism for \'etale fundamental group is surjective \cite[\'Expos\'e X, Corollaire 2.3]{SGA1}. 

In part $(ii)$ and $(iii)$,  \cite{Ch16} only deals with the \'etale cohomology groups, but the same argument holds for crystalline cohomology groups as well.  
\end{proof}

\begin{remark}
One can also consider the moduli space of semistable purely 1-dimensional sheaves and a similar result still holds. In particular, if $v=(0,c_1, s)$ satisfying $\left<v, v\right>=0$ and $\gcd(H\cdot c_1,s)=1$, then every  $H$-semistable coherent sheaf is $H$-stable and $M_H(S,v)$ is a (non-empty) smooth K3 surface ({\it cf}.~\cite[Proposition 4.20]{BL18}). 
\end{remark}

\begin{remark}
	When the K3 surface $S$ is ordinary (\textit{i.e.}~$\widehat{\Br}(S)$ is of height $1$), $\M_H(S,v)$ is indeed irreducible symplectic,  \textit{i.e.}~$h^{2,0}(\M_H(S,v))=1$, thanks to Charles \cite[Proposition 2.6]{Ch16}. His method of computing $h^{2,0}$ can not be applied to K3 surfaces whose formal Bauer group is of height $>1$, since in general we can only lift $(S, H, v)$ to a finite cover of $W$.
\end{remark}

\subsection{Moduli spaces of sheaves on supersingular K3 surfaces }
\label{subsect:ModuliSSK3}
Now we specialize the previous discussion to supersingular K3 surfaces. Our first goal is to show in Lemma \ref{lemma:ConditionStar} that in the supersingular case, the condition (\ref{eqn:GCD3}) can always be achieved, up to twisting by line bundles, except an extreme case. Let us first recall two standard auto-equivalences of the derived category ${\rm D^{b}}(S)$ of bounded complexes of coherent sheaves, as well as the corresponding cohomological transforms (\emph{cf.}~\cite{MR2244106}). For simplicity, we will use the same notation for a line bundle and its first Chern class.
\begin{enumerate}
\setlength{\itemindent}{-.1in}
	
\item Let $L$ be a line bundle, then tensoring with $L$ gives an equivalence
$$-\otimes L: {\rm D^{b}}(S)\xrightarrow{\simeq} {\rm D^{b}}(S).$$ The corresponding map on cohomology  $$\exp_L:\tilde{H}(S)\rightarrow\tilde{H}(S)$$  sends a Mukai vector $v=(r,c_1,s)$ to  $e^{L}\cdot v=(r,c_1+rL, s+\frac{r L^2}{2}+c_1\cdot L)$.

\item Let $E\in {\rm D^{b}}(S)$ be a spherical object, that is, $\dim \Ext^{*}(E, E)=\dim H^{*}(\mathbb{S}^{2}, \mathbb{Q})$, where $\mathbb{S}^2$ is the 2-dimensional sphere, whence the name. There is the spherical twist with respect to $E$ (see \cite{MR1831820}, \cite[\S\,8.1]{MR2244106})
$$T_{E}: {\rm D^{b}}(S)\xrightarrow{\simeq} {\rm D^{b}}(S),$$ and  we may still use $T_E$ to denote the the corresponding cohomological transform on the algebraic Mukai lattice.  In particular, when $E=\cO_C(-1)$ with $C\cong \PP^1$ a $(-2)$-curve on $S$,  we have $$T_E(r,c_1,s)=(r,s_{[C]} (c_1),s)$$ where $s_{[C]}$ is the Mukai reflection with respect to the $(-2)$-class $[C]\in \NS(S)$. 

\end{enumerate}

\begin{definition}\label{def:star}
Let $\Lambda$ be a lattice. Given an element $x\in \Lambda$ and $m\in \ZZ^{>0}$, we say that $x\cdot \Lambda$ is  \emph{divisible by} $m$ and denote it by $m\mid x\cdot \Lambda$  if  $m\mid (x,y)$ for all $y\in \Lambda$.  We denote by  $m\nmid x\cdot \Lambda$  if $x\cdot\Lambda$ is not divisible by $m$.  When $\Lambda=\NS(S)$, a Mukai vector $v=(r,c_1,s)\in \widetilde{H}(S)$ is called {\it coprime to $p$}, if 
 $p\nmid r$ or $p\nmid s$  or $p \nmid c_1\cdot \NS(S)$. 
\end{definition}

The following observation says that we can always reduce to the case where the numerical condition (\ref{eqn:GCD3}) holds, if the K3 surface is supersingular. 
\begin{lemma}\label{lemma:ConditionStar}
	Let $S$ be a supersingular K3 surface. If $v\in \tilde H(S)$ is a Mukai vector coprime to $p$, then up to a derived equivalence of tensoring with a line bundle, the numerical condition (\ref{eqn:GCD3}) and hence the conclusions in 
	Proposition \ref{msheaf} holds for any ample line bundle $H$. 
\end{lemma}	
\begin{proof}
	As tensoring a sheaf with a line bundle $L$ induces an isomorphism of moduli spaces 
	$$M_H(S,v)\cong M_H(S, \exp_L(v)),$$	
	it suffices to prove the existence of a line bundle $L$, such that $\exp_L(v)$ verifies the condition (\ref{eqn:GCD3}), \emph{i.e.}
	\begin{equation}
	\gcd(r,(c_1+rL)\cdot H, s+\frac{rL^2}{2}+c_1\cdot L)=1
	\end{equation}
We will make use of  Proposition \ref{NS-ss} on the structure of the lattice $\NS(S)$.	

Set $q=\gcd(r,c_1\cdot H) $ and we have $$\gcd(r,(c_1+rL)\cdot H, s+\frac{rL^2}{2}+c_1\cdot L)=\gcd(q,s+c_1\cdot L).$$ 
Let $q_{1}, \cdots, q_{r}$ be the distinct prime numbers dividing $q$. If for any $i$, we could manage to find a line bundle $L_{i}\in \NS(S)$ such that $\gcd(q_i, s+c_1\cdot L_i)=1$, \emph{i.e.} $q_{i}\nmid s+c_{1}\cdot L_{i}$, then the line bundle 
$$L:=\sum_{i=1}^{r}(a_{i}\prod_{j\neq i}q_{j})\cdot L_{i}$$ satisfies that $q_{i}\nmid s+c_{1}\cdot L$ for any $i$, or equivalently, $\gcd(q, s+c_{1}\cdot L)=1$, where $a_{i}$'s are integers such that $\sum_{i=1}^{r}a_{i}\prod_{j\neq i}q_{j}=1$.
	
As a result, one can assume that $q=\gcd(r,c_1\cdot H) $ is a prime number. Now, if $q=p$, then by the hypothesis that the Mukai vector $v$ is coprime to $p$, we know that either $p\nmid s$ or there exists $L_0\in \NS(S)$ such that $p\nmid c_1\cdot L_0$. Then we can take $L$ to be the trivial bundle or $L_0$ respectively.  If $q\neq p$ and assume by contradiction that $\gcd(q,s+c_1\cdot L)=q$ for all $L\in \NS(S)$. This means $$q\mid s~ \hbox{and}~q\mid c_1\cdot \NS(S).$$ Then, note that $\NS(S)$ is a $p$-primary lattice, the class $c_1$ has to be divisible by $q$. This is invalid because the vector $v=(r,c_1,s)$ is primitive by our assumption. 
\end{proof}

\begin{remark}\label{rmk:Exception}
 The exceptional case where the Mukai vector is not coprime to $p$ can only possibly happen when $p\mid \frac{1}{2}\left<v, v\right>$. This is obvious when $p>2$. When $p=2$, this follows from the fact that $\NS(S)$ is an even lattice of type I in the sense of Rudakov--{\v S}afarevi{\v c} \cite[\S\,2]{RS78}.   In other words, $M_H(S,v)$ is a smooth projective symplectic variety for any polarization $H$ whenever $M_H(S,v)\neq \emptyset$ and $p\nmid \frac{1}{2}(\dim M_H(S,v)-2) $.  
\end{remark} 

Next, we discuss the birational equivalences between moduli spaces $M_H(S,v)$ when varying the stability condition. The following result is based on Bayer--Macr\`i's wall-crossing principle in characteristic $0$ (see \cite{BM14}). 
\begin{theorem}\label{K3wall}
Let $S$ be a supersingular K3 surface over an algebraically closed field $k$ of characteristic $p>0$. Let $v_1,v_2\in \tilde{H}(S)$ be two Mukai vectors that are coprime to $p$ (Definition \ref{def:star}). Let $H$ and $H'$ be two ample line bundles on $S$. Suppose that $v_{1}$ and $v_{2}$ are differed by a cohomological transform induced by an auto-equivalence of ${\rm D^b}(S)$ of the following form:
	\begin{enumerate}
		\setlength{\itemindent}{-.1in}
		\item tensoring with a line bundle;
		\item spherical twist associated to a line bundle or to a sheaf of the form $\cO_C(-1)$ for some smooth rational curve $C$ on $S$; 
	\end{enumerate}
	Then $M_H(S,v_1)$ is liftably birational to $M_{H'}(S,v_2)$ in the sense of Definition \ref{def:LiftBir}.
\end{theorem}
\begin{proof}
This is indeed a special case of the wall-crossing principle for the moduli space of stable complexes on K3 surfaces, which is proved in characteristic zero. Over positive characteristic fields, the wall-crossing principle for K3 surfaces is not completely known.  But we can easily prove the assertion by lifting to characteristic zero. Here we briefly sketch the proof: 	using the argument in Proposition~\ref{msheaf}, we can find  a projective lift $(\cS, \cH,\cH',v_1, v_2)$  of the tuple $(S, H, H', v_1, v_2)$ over a discrete valuation ring $W'$ of characteristic zero whose residue field  $\kappa$ contains $k$.   This gives rise to projective liftings 
	\begin{equation}\label{proj-lift}
	M_{\cH}(\cS, v_1)\rightarrow \Spec W'\quad \text{ and }\quad M_{\cH'}(\cS, v_2)\rightarrow \Spec W'
	\end{equation}
	of  $M_{H}(S_\kappa, v_1)$ and $M_{H'}(S_\kappa, v_2)$ respectively as the moduli space of relative stable sheaves on $\cS$ over $W'$ with the given Mukai vectors.  
	
Let $K'$ be the fraction field of $W'$ and denote by $\cS_{K'}$ the generic fiber of $\cS\rightarrow \Spec W'$. In our case, we can assume that the auto-equivalences of type (1) or (2) can be lifted to $W'$ by lifting additional line bundles. By \cite[Corollary~1.3]{BM14}, after possibly taking a finite extension of $K'$,  there is a birational map
	\begin{equation}
	M_{\cH'}(\cS_{K'}, v_1)\dashrightarrow 	M_{\cH'}(\cS_{K'}, v_2)
	\end{equation}
between the generic fibers  of \eqref{proj-lift} defined over $K'$. This completes the proof.
\end{proof}

\begin{remark}\label{BMrem} \begin{enumerate}
\item The result
		\cite[Corollary 1.3]{BM14} is originally stated over complex numbers, but it is valid for any algebraically closed field of characteristic $0$. 
		\item 	In general,  one may consider the moduli space of Bridgeland stable complexes on K3 surfaces. According to the wall-crossing principle, they are expected to be birationally equivalent to the moduli space  of stable sheaves considered above. See \cite{MMY11} for more details. 
	\end{enumerate}
\end{remark}

\subsection{Relating moduli spaces to Hilbert schemes}
The goal of this subsection is to establish our main result Theorem~\ref{mainthm1} $(ii)$. The basic strategy is to exploit the elliptic fibration structure on supersingular K3 surfaces and to use the following result of Bridgeland.
\begin{theorem}[Bridgeland \cite{Br98}] \label{birationalK3}
Let $\pi:X\rightarrow C$ be a smooth relatively minimal elliptic surface over an algebraically closed field $\overline{K}$. We denote by $f\in \NS(X)$ the fiber class of $\pi$. Given $v=(r,c_1,s)\in \widetilde H(X)$ satisfying $r>0$ and
$\gcd(r, c_1\cdot f)=1$, then there exists an ample line bundle $H$ and a birational morphism 
\begin{equation}\label{ell-bir}
M_H(X,v)\dashrightarrow \Pic^0(Y)\times Y^{[n]} 
\end{equation}
where $Y$ is an smooth elliptic surface. 

Moreover, if $X$ is a K3 surface or an abelian surface, so is $Y$. The map \eqref{ell-bir} can be chosen liftable (Definition \ref{def:LiftBir}) when $\cha(\overline{K})>0$. In this case, the assertions still hold even when $\pi$ is only a quasi-elliptic fibration.
\end{theorem}

\begin{proof}
This theorem is due to Bridgeland  \cite{Br98} over the filed of complex numbers and his proof can be easily generalized to any algebraically closed field of characteristic $0$. See also \cite[Appendix]{Yo16} for a proof over arbitrary characteristic for elliptic surfaces without non-reduced fibers. 
	
We sketch here a simplified proof when $X$ is a K3 surface (\emph{resp.}~an abelian surface), by using the wall-crossing results in \cite{BM14} and \cite{MMY11}, and we refer the reader to \cite{Br98, Yo16} for more details.  As $\gcd(r, c_1\cdot f)=1$, we can set $(a,b)$ to be the unique pair of integers satisfying \begin{equation}br-a(c_1\cdot f)=1\end{equation} and $0<a<r$.  As in \cite{Br98}, there exists an ample line bundle $H\in \Pic(X)$ such that a torsion-free sheaf $\cF$ with Mukai vector $v$ is stable if and only if its restriction to the general fiber of $\pi$ is stable. 

Let $$Y:=M_{H'}(X, (0, af, b))$$ be  the fine moduli space of stable sheaves of pure dimension $1$ on $X$ with Mukai vector $(0,af, b)$, where $H'$ is a generic ample line bundle defined over $K$. Then $Y$ is a K3 surface (\emph{resp.} an abelian surface) as well.  Bridgeland's work \cite[Theorem 5.3]{Br98} essentially shows that the universal Poincar\'e sheaf  on $X\times Y$ induced  an equivalence $\Phi$ from $\rD^b(Y)$ to $\rD^b(X)$ with $\Phi_\ast(1,0,n)=v$. This enables us to conclude the assertion by using  \cite[Corollary 1.3]{BM14}. 
	
When $\cha(\overline{K})>0$, assume first $X$ is a K3 surface with a (quasi-)elliptic fibration. By the lifting result Proposition \ref{lift}, there are a complete discrete valuation ring $W'$ of characteristic zero with residue field $\kappa$ containing $\overline{K}$, a K3 surface $\cX$ over $W'$ with special fiber $X_\kappa$ such that the geometric generic fiber has an elliptic fibration. Hence, up to a finite extension of $W'$, there is a birational map 
\begin{equation}\label{fam2}
\begin{tikzcd}
	M_{\cH}(\cX, v)\arrow{dr}{~} \arrow[rr,dashrightarrow] &&  \Pic^0(\cY)\times \cY^{[n]}  \arrow{dl}{~}\\
	& \Spec W'
	\end{tikzcd}
	\end{equation}
where $\cY=M_{\cH'}(\cX, (0,af,b))$. The special fiber $Y_\kappa$ of $\cY\rightarrow \Spec W'$ is exactly $M_{H'}(X_\kappa, (0,af,b))$, which is a smooth K3 surface.  Then by taking their reduction, one can get   $M_H(X_\kappa,v)$ and $ \Pic^0(Y_\kappa)\times Y^{[n]}_\kappa $ are birationally equivalent. This proves the assertion. For abelian surfaces, as the lifting method is still valid (see Proposition \ref{prop:LiftingAb}), the same argument allows us to conclude.
\end{proof}

As mentioned before, we need a more refined analysis of (quasi-)elliptic fibrations on supersingular K3 surfaces. 
 The second key result is the following.

\begin{theorem}\label{ellstr} Let $S$ be a supersingular K3 surface defined over an algebraic closed field $k$ with characteristic $p>0$. Let $v=(r,c_1, s) \in \tilde{H}(S)$ be a Mukai vector that is coprime to $p$. 
Then up to changing the Mukai vector $v$ via the auto-equivalences listed in Theorem~\ref{K3wall}, there exists a (quasi-)elliptic fibration $\pi:S\rightarrow \PP^1$ such that $\gcd (r, c_1\cdot E)=1$, where $E\in\NS(S)$ is the fiber class of $\pi$. 
\end{theorem}
\begin{proof} 	We proceed the proof in several steps. 
\vspace{.1cm}

\noindent	{\bf Step 0.}
We claim that it suffices to show that there exists a primitive element $x\in \NS(S)$ such that
\begin{equation}\label{sq0}
x^2=0~\hbox{and}~\gcd (r, c_1\cdot x)=1.
\end{equation}
This is because the solution $x$ in \eqref{sq0} gives  an effective divisor $E$ on $S$ with $E^2=0$ and $c_1(E)=\pm x$. If the linear system $|E|$ is  base point free, then $|E|$ defines a (quasi-)elliptic fibration $S\rightarrow \PP^1$ as desired. If $E$ is not base point free, one can find a base point free divisor $E'$ with $(E')^2=0$  obtained by taking finitely many Mukai reflections of $E$ with respect to a sequence of $(-2)$ curves $C_1, C_2,\ldots, C_n$ on $S$, \textit{i.e.}~$E'=s_{[C_1]}\circ s_{[C_2]}\ldots \circ s_{[C_n]}(E)$.  Then we set $$v'=T_{\cO_{C_1} (-1)}\circ T_{\cO_{C_2} (-1)}\ldots T_{\cO_{C_n} (-1)}(v).$$ The Mukai vector $v'$ and the  (quasi-)elliptic fibration $S\rightarrow \PP^1$ induced by $|E'|$  thus  satisfy the desired condition. 
\vspace{.1cm}
	
\noindent {\bf Step 1.} We reduce to the case $c_1$ is primitive. First, if the vector $(r,c_1)\in \ZZ\oplus \NS(S)$ is primitive and $c_1$ is not primitive, then $c_1+rL$ will be primitive for some $L\in \NS(S)$. This means that we can replace $v$ by the Mukai vector $v'=\exp_L(v)$. Similarly, if $(c_1,s)$ is primitive, one can get $v'=T_L(v)$ satisfying the condition for some line bundle $L$. 
	
Suppose now the vectors $(r,c_1)$ and $(c_1, s)$ are both non-primitive, we can write 	
	\begin{equation}
	r=q_1r', ~c_1=q_1q_2c_1',  ~s=q_2s', 
	\end{equation}
	such that  $(r', q_2c_1')$ and $(q_1c_1', s')$ are primitive, where $r',q_i,s'\in\ZZ$ are non-zero integers and $c_1'\in \NS(S)$.  
	Then what we need is the following: find $L\in \NS(S)$ and $L^2=0$ such that the vector
	\begin{equation}
	(c_1+rL,   s+\frac{rL^2}{2} +c_1\cdot L)=(q_1(q_2c_1'+ r' L),  s+c_1\cdot L)
	\end{equation}
	is primitive because one can replace $v$ by $\exp_L(v)$ and the same argument as above works. To see the existence of such an $L$, note that $$(q_1, s+c_1\cdot L)=(q_1, q_2(s'+ q_1c_1'\cdot L ))$$ is always primitive,   it suffices to find a square zero element $L$ such that the class $ q_2c_1'+r'L$ is primitive. This can be easily achieved because $(r', q_2c_1')\in \ZZ\oplus \NS(S)$ is primitive by our assumption and the natural basis of $\NS(S)$ contains square zero element (See Proposition \ref{NS-ss}). 
	\vspace{.1cm}
	
\noindent{\bf Step 2.} Let us first assume that $p\nmid r$.  Suppose $r$ is a prime number. Recall that $\NS(S)$ is isomorphic to the lattice $-\Lambda_{\sigma(S)}$ given  in Proposition \ref{NS-ss}, we may have two cases as below: 
	
\noindent (1) If  $\Lambda_{\sigma(S)}$ contains a hyperbolic lattice $U$, we denote by  $\{ f_1, f_2\}$ the natural basis of  $U$ satisfying $f^2_1=f_2^2=0$ and $f_1\cdot f_2=1$.   We could find a basis of $\NS(S)$ of the form  the  \begin{equation}\label{basis}
	f_1,~f_2,~ v-\frac{\left<v, v\right>}{2}f_1+f_2 v 
	\end{equation}
	for some $v\in U^\perp$. 
	Note that every element in \eqref{basis} is square zero, so if there is no solution of \eqref{sq0},  we must have $r\mid c_1\cdot \NS(S).$  However,
	as $c_1$ is primitive and the  lattice $\NS(S)$ is $p$-primary,  it forces  $r=p$ which contradicts to our assumption. 
	
\noindent (2)	If  $\Lambda_{\sigma(S)}$ contains  $U(p)$ with basis $\{f_1,f_2\}$ satisfying $f_i^2=0, f_1\cdot f_2=p$,  the primitive square zero elements
	\begin{equation}\label{basis2}
	f_1, f_2, pv-\frac{p \left<v, v\right>}{2}f_1+f_2
	\end{equation}
	can also form a basis of the $p$-primary sublattice $$U(p)\oplus p (U(p)^\perp) \subseteq \NS(S).$$ 
	Then the same argument shows that there must exists $x\in U(p)\oplus p((U(p)^\perp)$ satisfying \eqref{sq0}. 
	
Combining (1) and (2), we prove the case when $r$ is a single prime. In general, for any positive integer $r$, let $\Xi\subseteq \ZZ$ be the collection of all prime factors of $r$. Let $X$ be the projective quadric hypersurface over $\ZZ$ defined by $x^2=0$, which is geometrically integral. We can view the desired element $x\in \NS(S)$ as a rational point in $X(\ZZ)$. By weak approximation on quadrics (\emph{cf.}~\cite[\S\,7.1 Corollary 1]{PR94}), the diagonal map
\begin{equation}\label{approx}
X(\ZZ)\rightarrow \prod_{q\in \Xi} X(\ZZ_q)\end{equation}
is dense if $X(\ZZ)\neq \emptyset$. Set $$U_q=\{x\in X(\ZZ_q)|~x\cdot c_1\in \ZZ_q^\times \}\subseteq X(\ZZ_q).$$
and we know that $U_q$ is an open subset of $X(\ZZ_q)$ if $U_q\neq \emptyset$. From the discussion above,  $U_q$ is non-empty  for each $q\in \Xi$ and hence it is an open  dense subset. Thus, there exists $x\in X(\ZZ)$ whose image via \eqref{approx} lies in $\prod\limits_{q\in\Xi} U_q$. Then we have  $\gcd(x\cdot c_1, q)=1$ for all $q\in \Xi$ from the construction, which  proves the assertion.  If $p\nmid s$, we can replace $v$ by $T_{\cO_S}(v)$. Then the same argument applies to $T_{\cO_S}(v)$. 
\vspace{.1cm}
	
	\noindent {\bf Step 3.} Finally, if $p\mid \gcd(r,s)$, then we have $p\nmid c_1\cdot \NS(S)$ by $(\ast)$.  As in Step 2, we only need to show $U(\ZZ_q)$ is not empty for all $q\in \Sigma$.   When $q\neq p$,  it is easy to see
	$U(\ZZ_q)\neq \emptyset$ via the same analysis.  
	
	When $q=p$, note that  $p\nmid c_1\cdot \NS(S)$,  the argument in {\bf Step~2} actually shows that there exists an element $x$ in \eqref{basis} or \eqref{basis2} such that $p\nmid c_1\cdot x$. This implies that $U_p$ is non-empty and hence proves the assertion. 
\end{proof}

\subsection{Proof of Theorem \ref{mainthm1}}
The assertion $(i)$ follows from Proposition \ref{msheaf} $(i)$ and $(iii)$.

To prove $(ii)$, assume that $S$ is supersingular. By Theorem~\ref{K3wall} and Theorem~\ref{ellstr}, there exists $v'=(r',c_1(L),s')\in \tilde{H}(S)$ for some $r', s'\in \ZZ$ and $L\in\Pic(S)$, and a (quasi-)elliptic fibration  $\pi:S\rightarrow \PP^1 $ such that  $X=M_H(S,v)$ is quasi-liftably birational to $M_H(S,v')$ and $\gcd(r', c_1(L)\cdot E)=1$, where $E$ is the fiber class of $\pi$.  Then one can invoke Bridgeland's result Theorem \ref{birationalK3} to see that $X$ is quasi-liftably birational to the $n$-th Hilbert scheme of some K3 surface $Y$ that is derived equivalent to $S$. As supersingular K3 surfaces have no Fourier--Mukai partners ({\it cf}.~\cite[Theorem~1.1]{LM11}), actually $Y\simeq S$. Therefore $X$ is quasi-liftably birational to the Hilbert scheme $S^{[n]}$. Since $S^{[n]}$ is irreducible symplectic (Beauville's proof \cite[\S~6]{MR730926} goes through in positive characteristics), and unirational if $S$ is, the same properties hold for $X$. 
	
For $(iii)$, thanks to Proposition \ref{prop:BirHK}, two quasi-liftably birational symplectic varieties have isomorphic Chow motives, hence it suffices to show that the Chow motive of the Hilbert scheme $S^{[n]}$ is of Tate type for a supersingular K3 surface $S$. 

To this end, we invoke the following result of de Cataldo--Migliorini \cite{DCM02} on the motivic decomposition of $S^{[n]}$: in the category of rational Chow motives,
$$\mathfrak h\left(S^{[n]}\right)(n)\simeq \bigoplus_{\lambda\dashv n}\mathfrak h\left(S^{|\lambda|}\right)^{\mathfrak S_{\lambda}}(|\lambda|),$$
where $\mathfrak h$ is the Chow motive functor, the direct sum is indexed by all partitions of $n$ and for a given partition $\lambda=(1^{a_{1}}2^{a_{2}}\cdots n^{a_{n}})$, its length $|\lambda|:=\sum_{i}a_{i}$ and $$\mathfrak S_{\lambda}:=\mathfrak S_{a_{1}}\times \cdots\times \mathfrak S_{a_{n}}.$$
As a result, we only need to show that the Chow motive of $S$ is of Tate type. This was proved for $p\geq 5$ by Fakhruddin  \cite{Fak02}. In general, if $S$ is unirational, there exists a surjective morphism $\widetilde{\PP^{2}}\to S$, where $\widetilde{\PP^{2}}$ is a successive blow-up of $\PP^{2}$ at points. Therefore $\h(\widetilde{\PP^{2}})$ is of Tate type by the blow-up formula. Hence $\h(S)$, being a direct summand of $\h(\widetilde{\PP^{2}})$, must be of Tate type, too. Since the unirationaly of supersingular K3 surfaces is only know for Kummer surfaces and remains open in general, we use the result proved by  Bragg and Lieblich instead. By \cite[Proposition 5.15]{BL18},  every supersingular K3 surfaces $S$ with Artin invariant $\sigma(S)$ is derived equivalent to a twisted K3 surface $S'$ with Artin invariant $\sigma(S')<\sigma(S)$ (see also \S\,5).  Thanks to \cite[Theorem 2.1]{Dan17} (extended in \cite[Theorem 1]{FuVialTorelli}), the Chow motives of $S$ and $S'$ are isomorphic.  Now by induction on the  Artin invariant,  one can show that the Chow motive of $S$ is isomorphic to the Chow motive of a supersingular K3 surface of Artin invariant 1, which is a Kummer surface. The Chow motive of a supersingular Kummer surface is known to be of Tate type, since it is unirational by Shioda \cite{MR0572983} (or one simply uses Corollary \ref{cor:SSAV}).

\section{Moduli spaces of twisted sheaves on K3 surfaces} \label{section5} 

In this section, we extend our results in \S\,\ref{sect:ModuliK3} to the moduli spaces of \emph{twisted} sheaves on K3 surfaces. 
\subsection{Twisted sheaves on K3 surfaces}
We mainly follow \cite{LMS14} and \cite{BL18} to review the basic facts of twisted sheaves on K3 surfaces. 
Let $S$ be a K3 surface over  $k$ and let $\srS\rightarrow S$ be a $\mu_m$-gerbe over $S$. This corresponds to a pair $(S,\alpha)$ for some  $\alpha\in H^2_{fl}(S,\mu_m)$, where the cohomology group is with respect to the flat topology. For any integer $m$, there is a Kummer exact sequence 
$$1\rightarrow \mu_m\rightarrow {\mathbb G}_{\mathrm{m}} \xrightarrow{x\mapsto x^m}  {\mathbb G}_{\mathrm{m}}\rightarrow 1$$
in flat topology and it  induces a surjective map 
\begin{equation}\label{braumap}
H^2_{fl}(S,\mu_m)\rightarrow \Br(S)[m]. 
\end{equation}

\begin{definition}[$B$-fields]
For a prime $\ell\neq p$, 	an  {\it $\ell$-adic $B$-field}  on $S$ is an element $B\in H^2_{\et}(S,\QQ_\ell(1))$.  It can be written as $\alpha/ \ell^n$  for some $\alpha\in H^2_{\et}(S,\ZZ_\ell(1))$ and $n\in\ZZ$.  We associate to it a Brauer class $[B_\alpha]$, defined as the image of $\alpha$ under the following composition of natural maps
\begin{equation*}
H^2_{\et}(S,\ZZ_\ell(1)) \rightarrow H^2(S,\mu_{\ell^n})\rightarrow \Br(S)[\ell^n],~\hbox{if $\ell\nmid p$}.
\end{equation*}
A {\it crystalline B-field} is an element $B=\frac{\alpha}{p^n}\in H^2_{\cris}(S/W)\otimes K$ with $\alpha\in H^2_{\cris}(S/W)$, so that the projection of  $\alpha$ in  $H^2_{\cris}(S/W_n(k))$ lies in the image of the map 
\begin{equation}\label{dlog}
H^2_{fl}(S,\mu_{p^n})\xrightarrow{d\log} H^2_{\cris}(S/W_n(k)).
\end{equation}
See \cite[I.3.2, II.5.1]{Il79} for the details of the map \eqref{dlog}. Then we can associate a $p^n$-torsion Brauer class $[B_\alpha]$ via the map \eqref{braumap}. 
\end{definition}

Let us write $B=\frac{\alpha}{r} $ as either a $\ell$-adic  or crystalline $B$-field of $\srS\rightarrow S$, we define the twisted Mukai lattice as 
\begin{equation}\label{Mlat}
\widetilde{H}(\srS)=\begin{cases}
e^{a/r} (\tilde{H}(S)\otimes \ZZ_\ell); \hbox{if} ~p\nmid m\\ ~\\ 
e^{a/r}(\tilde{H}_{\cris}(S/W)); \hbox{if}~m=p^n
\end{cases}
\end{equation}
under the Mukai pairing \eqref{pairing}. 

\begin{definition}
An $\srS$-{\it twisted sheaf} $\cF$ on $\srS$ is an $\cO_{\srS}$-module compatible with the $\mu_m$-gerbe structure ({\it cf}.~\cite[Def 2.1.2.4]{Lie07}). With the notation as above, the Mukai vector of $\cF$ is defined as 
$$v^{\alpha/r}(\cF)=e^{a/r}\ch_{\srS}(\cF) \sqrt{\td_S} \in \CH^\ast(S,\QQ). $$
where $e^{a/r}=(1,a/r, \frac{(a/r)^2}{2})$ and $\ch_{\srS}(\cF)  $ is the twisted Chern character of $\cF$ ({\it cf}.~\cite[3.3.4]{LMS14}). It  can be also viewed as an element in the twisted Mukai lattice $\widetilde{H}(\srS)$ via the corresponding cycle class map.
\end{definition}

Similarly as in the case of untwisted sheaves, we say that $H$ is {\it general with respect to $v$} if every $H$-semistable twisted sheaf is $H$-stable.  


\subsection{Moduli spaces of twisted sheaves}
\begin{definition}[{\it cf.}~\cite{Lie07}]
Fix a polarization $H$ on a K3 surface $S$, the moduli stack  $\srM_H(\srS, v)$ of $\srS$-twisted sheaves with Mukai vector $v$ is the stack whose objects over a $k$-scheme $T$ are pairs $(\cF, \phi)$, where $\cF$ is a $T$-flat quasi-coherent twisted sheaf of finite presentation and $\phi:\det \cF\rightarrow \cO(D)$ is an isomorphism of invertible sheaves on $X$, such that for every geometric point $t\rightarrow T$, the fiber sheaf $\cF_t$ has Mukai vector $v$ and endomorphism ring $k(t)$.
\end{definition}

\begin{theorem}\label{sstw}
Assume that  $\cha(k)=p>0$ and $v=(r, c_{1}, s)$ satisfies $\left<v, v\right>\geq 0$ and $r>0$. If $ H$ is general with respect to $v$, the moduli stack $\srM_H(\srS, v)$   has a smooth and projective  coarse moduli space $M_H(\srS,v)$ of dimension $\frac{\left<v, v\right>}{2}+1$ if non-empty.  The coarse moduli space $M_H (\srS, v)$  is a symplectic variety and  there exists a  canonical quadratic forms on  the N\'eron--Severi group $\NS(M_H(\srS, v))$ such that there is an injective isometry
\begin{equation}\label{twistNS}
	(v^\perp )\cap\widetilde{H}(\srS)  \rightarrow \NS(M_H(\srS, v))\otimes R.
\end{equation}
 where $R=\ZZ_\ell$  or $W$ depending on $\srS$ as in \eqref{Mlat}. Moreover, there is an isomorphism 
 \begin{equation}\label{Isocry}
 v^\perp\otimes K\rightarrow H^2_{\cris} (M_H(\srS, v)/K)
 \end{equation}
 as $F$-isocrystals. Here we regard $v$ as an element in $H^*(S/K)$ and $v^\perp\otimes K$ is the orthogonal complement with respect to the Mukai pairing.
\end{theorem}
\begin{proof}
Everything is known in characteristic $0$  by Yoshioka \cite[Theorem~3.16; Theorem~3.19]{Yo06}. Similar to the untwisted case,   as $H$ is general, all the assertion can be proved by lifting to characteristic $0$ (see \cite[Theorem~2.4]{Ch16}). We also refer to \cite{BL18} for the case where $S$ is supersingular. 
\end{proof}

Let us consider the case where $S$ is supersingular. By \cite{Ar74}, we know that the Brauer group $\Br(S)$ is of $p$-torsion.  In this case, there is an explicit description of twisted Mukai lattice  in \cite{BL18}.   This enables us to give a sufficient condition for  $M_H(\srS, v)$ to be smooth and projective.  The following result is  analogous to Lemma \ref{lemma:ConditionStar}. 

\begin{proposition}
Suppose $S$ is supersingular. If the Mukai vector $v$ is coprime to $p$ with $r>0$ and $\left<v, v\right>\geq 0$,  the coarse moduli space  $M_H(\srS,v)$ is a symplectic variety for any ample polarization $H$.  
\end{proposition}
\begin{proof}
The non-emptyness again follows from \cite[Theorem~3.16]{Yo06} (see also \cite[Proposition 4.20]{BL18}). It suffices to show that $H$ is general with respect to $v$. Since $v$ is coprime to $p$, up to change $c_1$ by twisting a line bundle,  we know that $H$ and $v$ satisfy the numerical condition \eqref{eqn:GCD3}  by using the same argument in the proof of Lemma \ref{lemma:ConditionStar}. 
\end{proof}

\subsection{From twisted sheaves to untwisted ones}
We show all our conjectures in the introduction for most moduli spaces of twisted sheaves on K3 surfaces.
The key result is Theorem~\ref{twistbir}, which shows that the moduli space $M_H(\srS, v)$ of twisted sheaves on $\srS$ is quasi-liftably birational to some moduli space of untwisted sheaves. To start, we have
\begin{lemma}\label{lemss}
Assume that $M_H(\srS, v)$ is a symplectic variety. Then it is $2^{nd}$-Artin supersingular if and only if $S$ is supersingular.  
\end{lemma}
\begin{proof}
Similar to the untwisted case, this follows from the isomorphism \eqref{Isocry}. 
\end{proof}

By Lemma \ref{lemss}, we are reduced to consider the  moduli space of  (semi)-stable twisted sheaves on supersingular K3 surfaces.  Let $\rD^{(1)}(\srS)$  be the bounded derived category of twisted sheaves on $\srS$. If the gerbe $\srS\rightarrow S$ is trivial, then $\rD^{(1)}(\srS)=\rD^b(S)$. The following result shows that every twisted supersingular K3 surface is derived equivalent to a untwisted supersingular K3 surface, provided that the Artin invariant is less than 10. 
\begin{theorem}[Untwisting]\label{twtountw}
If $\sigma(S)<10$,	there is a  Fourier--Mukai equivalence from   $\rD^{(1)}(\srS)$ to $\rD^b(S')$ for some supersingular K3 surface $S'$.  
\end{theorem}
\begin{proof}	
First, we claim that there exists  primitive vectors $\tau \in \tilde{H}(\srS)$ and $w\in\tilde{H}(\srS)$  satisfying the condition
	\begin{equation}
	\tau^2=0~~\hbox{and}~~ p\nmid \tau\cdot w.
	\end{equation}
Assuming this,  we  consider the moduli space $\srM_H(\srS, \tau)$ for an ample line bundle $H$. By \cite[Theorem~4.19]{BL18}, $\srM_H(\srS, \tau)$  is a $\GG_m$-gerbe over a supersingular K3 surface $S'=M_H(\srS,\tau)$ if non-empty.  The universal sheaf induces a Fourier-Mukai equivalence 
$$\Phi: \rD^{(1)}(\srS)\rightarrow \rD^{(-1)}(\srM_H(\srS,\tau) ).$$
So it suffices to show that the gerbe $\srM_H(\srS,\tau) \rightarrow S'$ is trivial. By \cite[Theorem~4.19 (3)]{BL18}, this is equivalent to find a vector $w\in \tilde{H}(\srS)$ such that $\tau \cdot w$ is coprime to $p$. Thus we can conclude our assertion.    
	
To prove the claim,  write $\alpha=\alpha_0+L$ for some $L\in \NS(S)$ and  $[\alpha_0]$ the image of $\alpha_0$ in $\Br(S)$.   There are two possibilities: 
\vspace{.1cm}

\textit{(i)}. If $L\cdot \NS(S)$ is not divisible by $p$, then we can find a primitive element $E\in \NS(S)$ such that $E^2=0$ and $p\nmid L\cdot E$.   Then we take $$\tau=(0, [E], s)~\hbox{and}~w=(p, L, \frac{L^2}{2p} ).$$  It follows that $\tau \cdot w=L\cdot E -sp$ is coprime to $p$. 
\vspace{.1cm}

\textit{(ii)}. If $p\mid L\cdot \NS(S)$, then we can take $\tau=(p, c_1,s)$ with $c_1^2=2ps$ and choose $w=( 0, D, \frac{D\cdot L}{p})$ for some $D$ with $p\nmid c_1\cdot D$. It follows that $$\tau\cdot w= c_1\cdot D-D\cdot L$$ is coprime to $p$.  The existence of $c_1$ and $D$ can be also deduced  from Proposition \ref{NS-ss}.  When $p>2$,  as $\sigma(S)<10$ and $\NS(S)$ contains a hyperbolic lattice $U$, we can let $f_i,i=1,2$ be the standard basis of $U$ and take $$c_1=f_1+pf_2~\hbox{and}~D=f_2$$
as desired.  

When $p=2$, this is more complicated. The lattice $\NS(S)$ contains $U$ only when $\sigma(S)$ is odd. 
If $\sigma(S)$ is even, $\NS(S)$ contains $U(2)$ as a direct summand instead.   We claim that there exist an element $y\in (U(2))^\perp$ such that $y^2=-4$ and  $2\nmid y\cdot \NS(S).$ Assuming this, we can pick $x\in U(2)$ with $x^2=4$ and $c_1\in \NS(S)$ with $2\nmid c_1\cdot y$,   one can easily see that the classes $D=x-y$ and $c_1$ are as desired.  Then one can use the explicit description of $\NS(S)$ to prove the claim. For instance,  if $\sigma(S)=8$, Rudakov and {\v S}afarevi{\v c} \cite{RS78} showed that $\Lambda_{\sigma(S)}$ is isomorphic to $$U(2)\oplus D_4^{\oplus 3}\oplus E_8(2),$$
where $D_4$ and  $E_8$ are  root lattices defined by the corresponding Dynkin diagram. 
There certainly exist some element $y'\in D_4$ with $(y')^2=2$ and $2\nmid y'\cdot D_4$. Then we can select $y=(y',y')\in D_4^{\oplus 2}$ with $y^2=4$ and $2\nmid y\cdot D_4^{\oplus 2}$  which automatically satisfies the conditions. Similar analysis holds when $\sigma(S)=2,4$ and $6$. Here we omit the details and left it to the readers. 
\end{proof}

\begin{remark}
Theorem \ref{twtountw} can be viewed as a converse of \cite[Proposition~5.15]{BL18}, which says that every supersingular K3 surface is derived equivalent to a supersingular twisted K3 surface with Artin invariant less than $10$. 
\end{remark}

\begin{theorem}\label{twistbir}
Let $v$ be a Mukai vector which is coprime to $p$. If the Artin invariant $\sigma(S)<10$,  then	$M_H(\srS,v )$ is  quasi-liftably birational to $M_{H'}(S',v')$ for some supersingular K3 surface $S'$, $v'\in \tilde{H}(S')$ a Mukai vector which is coprime to $p$ and $H'\in \Pic(S')$ an ample line bundle. 
\end{theorem}
\begin{proof}
By assumption, as in the proof of Theorem~ \ref{twtountw}, we can find $\tau \in \widetilde{H}(\srS)$ with $\tau^2=0$ such that there is a Fourier-Mukai equivalence
$$\Phi: \rD^{(1)}(\srS)\rightarrow  \rD^b(S'), $$ 	
where $S'=M_H(\srS,\tau)$. As before, we prove the birational equivalence  from the wall-crossing principle.  As this remain unknown over positive characteristic fields, we proceed the proof by lifting to characteristic $0$. As in Theorem \ref{sstw} (See also \cite[Theorem~4.19]{BL18}), we take a projective lift 
	\begin{equation}\label{mod0}
	\srS_W\rightarrow \cS_W, \cH\in \Pic(\cS),~\tau_W\in \tilde{H}(\srS_W),
	\end{equation}
	of the triple $(\srS_W\rightarrow S, H, \tau)$ over $W$ for some discrete valuation ring $W$.  Consider the relative moduli space 
	\begin{equation}\label{mod1}
	\srM_{\cH}(\srS, \tau_W)\rightarrow \Spec(W) ,
	\end{equation}
	whose special fiber  is the (untwisted) supersingular K3 surface $M_H(S,\tau)$ shown in Theorem~ \ref{twtountw}. By  \cite[Theorem~4.19]{BL18}, we have a Fourier-Mukai equivalence 
	\begin{equation}
\Phi:	\rD^{(1)}(\srS_\eta)\rightarrow \rD^{(1)}(\srM_\cH(\srS_\eta,\tau_\eta)),
	\end{equation}
	for the generic fiber of \eqref{mod0} and \eqref{mod1}. Due to the wall-crossing theorem,  there is a  birational map between the coarse moduli spaces
	\begin{equation}
	M_{\cH}(\srS_\eta, v)\dashrightarrow M_{\cH'}(\srS'_\eta,v'),
	\end{equation}
	after possibly taking a finite extension of $W$, where $v'=\Phi_\ast(v)$. Then we can conclude the assertion by taking the reduction.

Finally, we have to check that the Mukai vector $v'$ is coprime to $p$. This is clear as the Fourier--Mukai transform does not change the divisibility of Mukai vectors. 
\end{proof}

\begin{corollary} \label{cor:TwistedK3}
If the Artin invariant $\sigma(S)<10$, then the moduli space $M_H(\srS,v)$ is $2^{nd}$-Artin supersingular and all the same statements in Theorem \ref{mainthm1} hold for $M_H(\srS,v)$ as well.
\end{corollary}
\begin{proof}
As quasi-liftably birational symplectic varieties have isomorphic Chow motives by Proposition \ref{prop:BirHK}, Theorem \ref{twistbir} allows us to reduce to the untwisted case, namely, Theorem \ref{mainthm1}.
\end{proof}

\section{Moduli spaces of sheaves on abelian surfaces}\label{sect:Kummer}

\subsection{Preliminaries on abelian surfaces in positive characteristics}
Let $A$ be an abelian surface defined over an algebraically closed field $k$ of positive characteristic $p$. 
For $n\in\ZZ$, consider the  multiplication-by-$n$ map $$n_A:A\rightarrow A,$$ which is separable if and only if $p\nmid n$. When $p\mid n$, the inseparable morphism factors through the absolute Frobenius map $F:A\rightarrow A$. Denote by $A[n]$ the kernel of $n_A$. 
The Newton polygon of $H^1(A)$ can be computed via the $k$-rational points of $A[p]$ and in particular, 
\begin{itemize}
	\item $A$ is ordinary if $A[p](k)\cong (\ZZ/p\ZZ)^{\oplus 2}$;
	\item $A$ is supersingular if $A[p](k)=\{0\}$.
\end{itemize}
Note that the above characterizations no longer hold for higher-dimensional abelian varieties. 

When $A$ is a supersingular abelian surface, the N\'eron--Severi group $\NS(A)$ equipped with the intersection form, is an even lattice of rank $6$ with discriminant equal to $p^{2}$ or $p^{4}$ (one could say that the Artin invariant is 1 or 2). As before, we have the  following description of the $\NS(A)$, in a parallel way to Proposition \ref{NS-ss}. 

\begin{proposition}[{\cite[\S\,2]{RS78}, \cite{Sh04}}]\label{Prop:NS-ssAb}
	Let $A$ be a supersingular abelian surface defined over an algebraically closed field of positive characteristic $p$. 	If ${\rm disc} ( \NS(A)) =p^{2}$, then the lattice $\NS(A)$ is isomorphic to $-\Lambda_{1}$, with 
	\begin{equation}\label{ssAb1}
	\Lambda_{1}=\begin{cases} U\oplus D_4 & \text{if} ~p=2,\\
	U\oplus V_{4,2}^{(p)}, &\text{if}~ p\equiv 3\mod 4, \\ U\oplus H^{(p)} , &\text{if}~ p\equiv 1\mod 4,
	\end{cases}
	\end{equation}
	If ${\rm disc} (\NS(A))=p^{4}$, then the lattice $\NS(A)$ is isomorphic to $-\Lambda_{2}$, with
	\begin{equation}\label{ssAb2}
	\Lambda_{2}= \begin{cases} U(p)\oplus D_4 & \text{if} ~p=2,\\
U(p)\oplus H^{(p)}, &\text{otherwise}.
	\end{cases}
	\end{equation}
Here, $U,~V_{4,2}^{(p)}$ and $H^{(p)}$ are the lattices defined in Proposition \ref{NS-ss}. 
\end{proposition}

We have the following simple observation.
\begin{lemma}\label{pp}
	Any supersingular abelian surface admits a principal polarization.
\end{lemma}	
\begin{proof} Let $A$ be a supersingular abelian surface. By Proposition \ref{Prop:NS-ssAb},  we know that  the N\'eron--Severi lattice  $\NS(A)$ has two possibilities $-\Lambda_1$ or $-\Lambda_2$ given by \eqref{ssAb1} and \eqref{ssAb2}. In any case, we have a line bundle $L$ with $(L^{2})=2$. Replacing $L$ by its inverse if necessary, we see that $L$ is ample. Since $h^{0}(A,L)=\chi(A, L)=\frac{(L^{2})}{2}=1$, it is a principal polarization.
\end{proof}

Similarly to \S~\ref{subsect:ModuliSSK3}, we need the following two types of auto-equivalences of the derived category $\rD^b(A)$. Again, the same notation is used for a line bundle and its first Chern class.
\begin{enumerate}
    \item For any line bundle $L$ on $A$, tensoring with $L$ is an auto-equivalence
    \[-\otimes L\colon \rD^b(A)\xrightarrow{\simeq} \rD^b(A).\] The corresponding cohomological transform 
  $$\exp_L:\tilde{H}(A)\rightarrow\tilde{H}(A)$$  sends a Mukai vector $v=(r,c_1,s)$ to  $e^{L}\cdot v=(r,c_1+rL, s+\frac{r L^2}{2}+c_1\cdot L)$. 
\item Choose any principal polarization of $A$ (Lemma \ref{pp}), which allows us to identify $A$ and $\widehat{A}$. Let $\mathcal{P}$ be the Poincar\'e line bundle on $A\times A$. Mukai's Fourier transform \cite{MukaiFourier} gives an auto-equivalence:
\begin{align*}
\Phi\colon \rD^b(A)&\xrightarrow{\simeq} \rD^b(A)\\
E &\mapsto Rp_{2,*}(p_1^*(E)\otimes \mathcal{P}),
\end{align*}
where $p_1, p_2$ are natural projections from $A\times A$ to the factors. The cohomological Fourier transform was computed by Beauville \cite[Proposition 1]{MR726428}: it sends a Mukai vector $(r, c_1, s)$ to $(s, -c_1, r)$.
\end{enumerate}

As an analogue of Theorem \ref{ellstr} in the case of supersingular abelian surfaces, the following existence theorem of elliptic fibrations is a crucial step in the proof of Theorem \ref{mainthm2}.

\begin{theorem}\label{ellstr:ab}
Given a supersingular abelian surface $A$,  we denote by $\widetilde H(A)$ the \emph{(algebraic) Mukai lattice}.  Let $v=(r,c_1, s) \in \tilde{H}(A)$ be a Mukai vector that is coprime to $p$.  Then up to changing the Mukai vector $v$ via the two types of auto-equivalences recalled above,  there exists an elliptic fibration $\pi:A\rightarrow E$ such that $\gcd (r, c_1\cdot f)=1$  where $f\in\NS(A)$ is the fiber class of $\pi$ and $E$ is an elliptic curve. 
\end{theorem}

\begin{proof}
The existence of elliptic fibration is equivalent to the existence of a square zero element in $\NS(A)$ satisfying the coprime condition (note that there is no $(-2)$ curve on an abelian surface).  The argument is similar to that of  Theorem \ref{ellstr} with the following changes:
\begin{itemize}
    \item replace $\NS(S)$ by $\NS(A)$;
    \item skip Step 0;
    \item in Step 2 of the proof of Theorem \ref{ellstr}, we used spherical twists to ``changing the roles'' of $r$ and $s$, more precisely, to switch between the Mukai vectors $(r, c_1, s)$ and $(s, -c_1, r)$. Here for abelian surfaces, the structural sheaf is no longer spherical; instead, we use Mukai's Fourier transform $\Phi$ recalled above, which has the same effect on Mukai vectors. \qedhere
\end{itemize} 
\end{proof}

Next, we need to slightly generalize the lifting results of Mumford, Norman and Oort (\emph{cf.}~\cite{NO80}) for the liftability of supersingular abelian surfaces together with line bundles. The following result is analogous to  Proposition~\ref{lift}. 

\begin{proposition}\label{prop:LiftingAb}
Let $A$ be an abelian surface over a perfect field $k$ of characteristic $p>0$. Suppose  $L_1, L_2$ are two line bundles on $A$ with $L_1$  a separable polarization. Then there exists a complete discrete valuation ring  $W'$ of characteristic zero which is finite over $W (k)$, and an abelian scheme that is a projective lift of $A$
\begin{equation}
	\cA   \rightarrow \Spec(W'),
\end{equation}
such that $\rank ~\NS(\cA_{\eta})=2$ and the image of the specialization map  
\begin{equation}
\NS(\cA_\eta) \rightarrow \NS(A)
\end{equation}
contains $L_1, L_{2}$, where $\cA_\eta$ is the generic fiber  of $\cA$ over $W'$. In particular, every supersingular abelian surface $A$ admits a projective lift over $W$ such that $A$ is isogenous to the product of elliptic curves.
\end{proposition}

\begin{proof}
Similar to the proof of Proposition \ref{lift}, we first consider the case that $A$ is not supersingular.  As in \cite{LM11}, let ${\rm Def}(A; L_1,  L_2)$ be the deformation functor parametrizing deformations of $A$ together with $L_1$ and $L_2$.  For simplicity, we can assume that $L_2$ is a separable polarization.  By deformation theory of abelian varieties, the formal deformation space ${\rm Def}(A)$ of $A$ over $W$ is isomorphic to $$\Spf(W[[t_1,t_2,t_3,t_4]]).$$  As inspired by Grothendieck and Mumford,  each $L_i$ imposes one equation $f_i$ on $W[[t_1,t_2,t_3,t_4]]$ (\emph{cf.}~\cite[\S\,2.3-2.4]{Oo71}) and the forgetful functor ${\rm Def}(A;L_1,L_2)\rightarrow {\rm Def}(A)$ 
	can be identified as the quotient map 
	$$W[[t_1,t_2,t_3,t_4]]\rightarrow W[[t_1,t_2,t_3,t_4]]/(f_1,f_2). $$
	Then it is easy to see that 
	${\rm Def}(A;L_1,L_2)$ is formally smooth over $W$ (\emph{cf.}~\cite[2.4.1]{Oo71} and \cite[Proposition 4.1]{LM11}).  Since $W$ is Henselian, the $k$-valued point $(A; L_1, L_2)$ extends to a $W$-valued point, giving a formal lifting.  Moreover,  the formal family is formally projective. By the Grothendieck Existence Theorem,  this lifting is therefore algebraizable as a projective scheme, as desired.

If $A$ is supersingular, as in \cite[Lemma A.4]{LO15}, it suffices to show each triple $(A, L_1, L_2)$ is the specialization of an object $(A', L_1', L_2')$ of finite height along a local ring. In other words, we can deform an abelian surface in equal-characteristic to a non-supersingular  abelian surface.  Consider the formal universal deformation space  $$\Delta_A=\Spec k [[ x_1, \ldots , x_4]] $$ of $A$ over $k$. Similarly as above, each line bundle  $L_i$ determines a divisor in $\Delta_A$. The complete local ring of $\Delta_A$  at $(X,L_1,L_2) $ is given by two equations $f_1,f_2$.  Moreover, since $L_1$ is a separable polarization, the universal deformation $(\cA,\cL_1,\cL_2)$ of the triple  $(A,L_1,L_2)$ is 
	$$\Spf k[[t_1,\ldots, t_4]]/(f_1,f_2)$$ 
and it is algebraizable. Now, it is known that the supersingular locus in the  deformations of $A$ has dimension $1$ ({\it cf}.~ \cite{LO98}). But the deformation space $$T:=\Spec k[[t_1,\ldots, t_4]]/ (f_1,f_2)$$ has dimension at least $2$ and hence can not lie entirely in the supersingular locus. The generic point of $T$ parametrizes a triple $(A', L_1',L_2')$ with $A'$ of finite height. This proves the assertion. 
\end{proof}

\subsection{Generalized Kummer varieties}Let $A$ be an abelian surface over $k$ and $s:A^{[n+1]}\rightarrow A$ be the morphism induced by the additive structure on $A$, which is an isotrivial fibration.  By definition, the \emph{generalized Kummer variety} (see \cite{MR730926}) is its fiber over the origin:
$$K_{n}(A):= s^{-1}(O_{A}),$$ which is an integral variety of dimension $2n$ with trivial dualizing sheaf. It will be a symplectic variety if it is smooth. We shall remark that different from the case of characteristic zero,  the generalized Kummer varieties over positive characteristic fields can be singular and even non-normal (see \cite{Sc09} for some examples in characteristic $2$). In \cite{Sc09}, Schr\"oer raised the question when $K_n(A)$ is smooth. Here we partially answer this question.  

\begin{proposition}
$K_n(A)$ is a smooth symplectic variety if  $p\nmid (n+1)$. 
\end{proposition}
\begin{proof}
	We first show the smoothness. After the base change $(n+1)_A:A\rightarrow A$, we obtain a trivialization of $s:A^{[n+1]}\rightarrow A$, \textit{i.e.}~there is a cartesian commutative diagram
	\begin{equation}
	\xymatrix{
		A\times K_n(A) \ar[r]^-{\psi_n} \ar[d] &A^{[n+1]}  \ar[d]^s \\ A\ar[r]^{\times(n+1)} & A}
	\end{equation}
Since $p\nmid (n+1)$, the map $\psi_n$ is \'etale. Hence the smoothness of $A^{[n+1]}$ implies that $K_n(A)$ is smooth as well.

To show that $K_{n}(A)$ is a symplectic variety, we first use the lifting argument as in Proposition \ref{msheaf}: by lifting $A$ to an abelian scheme $\mathcal{A}$ over a base of characteristic zero $W$, we obtain that the relative generalized Kummer variety  $\mathcal{K}_n(\mathcal{A})$ is a lifting of $K_n(A)$ over $W$. The simple connectedness of $K_n(A)$ is obtained from the simple connectedness of the generic fiber of $\mathcal{K}_n(\mathcal{A})$ (a result known in characteristic zero) together with the surjectivity of the specialization map of \'etale fundamental groups \cite[Expos\'e X, Corollaire 2.3]{SGA1}.

Finally, let us construct a symplectic form on $K_n(A)$. Fix a symplectic/canonical form $\omega$ on $A$.  We perform the construction of Beauville \cite[Proposition 7]{MR730926} to get a symplectic form on $A^{[n+1]}$, which restricts to an algebraic 2-form $\omega_n$ on $K_n(A)$. We want to show that if $p\nmid (n+1)$, then $\omega_n$ is nowhere degenerate,
or equivalently, $\omega_n^{\wedge n}$ is nowhere vanishing.
Since the canonical bundle of $K_n(A)$ is trivial,
it suffices to show that $\omega_n^{\wedge n}$ is non-zero at some point.
At a point $\xi\in K_n(A)$ represented by an $(n+1)$-tuple of distinct points $\{x_0, \dots, x_n\}$ with $\sum_{i=0}^{n} x_i=0 \in A$, the tangent space $T_\xi K_n(A)$ is canonically identified with $\ker\left(T_0A^{\oplus n+1}\xrightarrow{+} T_0A\right)$ and the 2-form $\omega_n$ at $\xi$ is the restriction of the symplectic form $(\omega, \dots, \omega)$ from $T_0A^{\oplus n+1}$. The elementary lemma below shows that $\omega_n$ is non-degenerate at $\xi$, when $p\nmid (n+1)$.
\end{proof}

\begin{lemma}
Let $(V, \omega)$ be a symplectic vector space over a field of characteristic $p$ and $m\in \NN$. Let $(\omega, \dots, \omega)$ be the induced symplectic form on $V^{\oplus m}$. Then the restriction of $(\omega, \dots, \omega)$ to $V^m_0:=\ker\left(V^{\oplus m}\xrightarrow{+} V\right)$ is non-degenerate if and only if $p\nmid m$.
\end{lemma}
\begin{proof}
If $p\mid m$, then for any non-zero vector $v\in V$, we have $(v, \dots, v)\in V_0^m$ is in the kernel of the restricted 2-form: for any $(v_1, \dots, v_m)\in V_0^m$
\[\langle (v, \dots, v), (v_1, \dots, v_m) \rangle=\sum_{i=1}^m\langle v, v_i\rangle=\langle v, \sum_{i=1}^m v_i\rangle=0.\]

If $p\nmid m$, then we have the following section of the surjection $V^{\oplus m}\xrightarrow{+} V$,
\begin{align*}
    V &\to V^{\oplus m}\\
    v&\mapsto \left(\frac{1}{m}v, \dots, \frac{1}{m}v\right),
\end{align*}
which respects the symplectic forms and gives rise to an orthogonal direct sum decomposition $V^{\oplus m}=V\oplus V_0^m$. In particular, $V_0^m$ equipped with the restricted 2-form is a symplectic space.
\end{proof}

We have the following result concerning the motive of generalized Kummer varieties. 
\begin{proposition}\label{prop:SSKummer}  
	Let $K_{n}(A)$ be a smooth generalized Kummer variety associated to an abelian surface $A$.  Then $K_n(A)$ is $2^{nd}$-Artin supersingular if and only if $A$ is supersingular. Moreover,   the supersingular abelian  motive conjecture (see Conjeture A) and the supersingular Bloch--Beilinson--Beauville conjecture (see Conjecture B) hold for $K_n(A)$.  

\end{proposition}
\begin{proof} Since the standard conjecture is known for generalized Kummer varieties, the notion of $2^{nd}$-Shioda supersingularity is independent of the cohomology theory used. Let us write $X=K_{n}(A)$. For the first assertion,  since $H^{2}(X)\cong H^{2}(A)\oplus W(-1)$, the supersingularity of the $F$-crystal $H^{2}(X)$ is equivalent to the supersingularity of the crystal $H^{2}(A)$, which is equivalent to the supersingularity of $A$.

To prove the supersingular abelian motive conjecture, we use the following motivic decomposition of $K_{n}(A)$ (see \cite{DCM04}, \cite[Theorem 7.9]{FuTianVial2018} and \cite{XuZeKummer}):
	$$\frh(X)(n)\cong \bigoplus_{\lambda \dashv (n+1)} \frh(A_{0}^{\lambda})(|\lambda|),$$
	where $\lambda$ runs through all partitions of $n+1$; for a partition $\lambda=(\lambda_{1}, \cdots, \lambda_{l})$,  $|\lambda|:=l$ denotes its \emph{length} and 
	$$A_{0}^{\lambda}:= \left\{(x_{1}, \cdots, x_{l})\in A^{l}~~\vert~~ 
	\sum_{i}\lambda_{i}x_{i}=O_{A}\right\}.$$
	Observe that $A_{0}^{\lambda}$ is isomorphic, as algebraic varieties, to the disjoint union of $\gcd(\lambda)^{4}$ copies of the abelian variety $A^{l-1}$, where $\gcd(\lambda)$ is the greatest common divisor of $\lambda_{1}, \cdots, \lambda_{l}$. As a result, the motive of $X$ is a direct sum of the motives of some powers of $A$ with Tate twists, precisely:
	$$\frh(X)(n)\cong \bigoplus_{i} \frh(A^{l_{i}-1})(l_{i}).$$
	Since $A$ is supersingular, $\frh(X)$ is by definition a supersingular abelian motive. The fully supersingular Bloch--Beilinson conjecture for $X$ follows from Corollary \ref{cor:SSAV}. 
	
	To establish the supersingular Bloch--Beilinson--Beauville conjecture for $X$, it remains to show the section property conjecture for it. But this is done in Fu--Vial \cite[\S\,5.5.2]{FuVial17}, where the argument works equally in positive characteristics.
\end{proof}

Finally, we show the RCC Conjecture for the generalized Kummer varieties:
\begin{proposition}\label{prop:RCCKummer}
Let $A$ be a supersingular abelian surface and $n$ a natural number, then the generalized Kummer variety $K_{n}(A)$ is rationally chain connected.
\end{proposition}
\begin{proof}
The main input is Shioda's theorem \cite{MR0572983} that the Kummer K3 surface $K_{1}(A)$ is unirational.	Let us first consider the singular model of $K_{n}(A)$, namely $$K'_{n}(A):=A^{n+1}_{0}/\fS_{n+1},$$ where $A^{n+1}_{0}:=\left\{(x_{0}, \cdots, x_{n})\in A^{n+1}~~\vert~~ 
\sum_{i}x_{i}=O\right\}$ equipped with the natural $\fS_{n+1}$-action by permutation. A typical element of $K'_{n}(A)$ is thus denoted by $\{x_{0}, \cdots, x_{n}\}$.
	
Now we show that any two points of $K'_{n}(A)$ can be connected by a chain of unirational surfaces $K'_{1}(A)=A/{-1}$. Indeed, for any $\{x_{0}, \cdots, x_{n}\}\in K'_{n}(A)$ (hence $\sum_{i}x_{i}=0$), let us explain how it is connected to $\{O, \cdots, O\}$.
Firstly, the image of the map
\begin{eqnarray*}
		\varphi_1:	A/{-1}&\to& K^{'}_{n}(A)\\
		t &\mapsto& \{t,-t,O, \cdots, O\}
\end{eqnarray*}
connects $\{O,O,O, \cdots, O\}$ and$ \{x_{0},-x_{0},O, \cdots, O\}$. Next, one can choose any $u\in A$ such that $2u=x_{0}$, then the surface
\begin{eqnarray*}
		A/{-1}&\to& K^{'}_{n}(A)\\
		t &\mapsto& \{x_{0},t-u,-t-u,O, \cdots, O\}
\end{eqnarray*}
connects the  two points:  
\begin{center}$\{x_{0},-x_{0},O, \cdots, O\}$ and  $\{x_{0}, x_{1}, -x_{0}-x_{1}, O, \cdots, O\}$\end{center}
by taking $t=-u$  and $t=u+x_{1}$ respectively. Continuing this process $n$ times, we connect $\{O, \cdots, O\}$ to $\{x_{0}, \cdots, x_{n}\}$ in $ K'_{n}(A)$ by the unirational surfaces $K'_{1}(A)$. In conclusion, $K'_{n}(A)$ is rationally chain connected. At last, note that we have a birational morphism 
\begin{equation}\label{crep}
	K_n(A)\rightarrow K'_n(A)
\end{equation}
as the crepant resolution. As the exceptional divisors of \eqref{crep} are rationally chain connected (see also Remark \ref{RCCbir}), it follows that $K_n(A)$ is rationally chain connected as well. 
\end{proof}

\subsection{Moduli spaces of stable sheaves on abelian surfaces} Now we turn to the study of moduli spaces of sheaves on abelian surfaces and their Albanese fibers in general. 
Given $v\in \widetilde H(A)$ and a general polarization $H$ with respect to $v$, we denote by $M_{H}(A,v)$  the moduli space of stable sheaves on $A$ with Mukai vector $v$. Choosing $F_0\in M_{H}(A,v)$, the Albanese morphism 
\begin{align*}
   \alb\colon M_{H}(A,v)&\to \Pic^{0}(A)\times A=\widehat{A}\times A,\\
  F &\mapsto (\det(F)\otimes \det(F_0)^{-1}, \alb(c_2(F)-c_2(F_0)),
\end{align*}
is an isotrivial fibration. The Albanese fiber is denoted by $K_{H}(v)$. If $K_{H}(v)$ is smooth, it is of dimension $2n:=\left<v, v\right>-2$, and is deformation equivalent to the generalized Kummer variety $K_n(A)$.  The following result is similar to Proposition \ref{msheaf} for K3 surfaces. 


\begin{proposition}\label{msheaf:ab}
Let $A$ be an abelian surface, $H$ a polarization and $v=(r, c_{1}, s)\in \widetilde{H}(A)$ a primitive element such that $r>0$ and $\left<v,v\right>\geq 2$. Denote $2n:=\left<v, v\right>-2$. Assume that $p\nmid (n+1)$. If $H$ is general with respect to $v$, then
	
\begin{enumerate}[(i)]
		\item $M_H(A,v)$ is a smooth projective variety of dimension $\left<v, v\right>+2$ over $k$. Let $K_H(A,v)$ be the fiber of its Albanese map
		\begin{equation}\label{albmap}
		M_H(A,v)\rightarrow A\times \widehat{A}.
		\end{equation}
Then $K_H(A, v)$ is a smooth projective symplectic variety of dimension $2n$ and deformation equivalent to the generalized Kummer variety $K_{n}(A)$. 
		\item  When $\ell \neq p$,  there is a canonical quadratic form on $H^2(K_H(A,v), \ZZ_\ell(1))$.   Let $v^\perp$ be the orthogonal complement of $v$ in  the $\ell$-adic Mukai lattice of $A$.  There is an injective isometry
		$$\theta_v: v^\perp\cap \widetilde H(A)\rightarrow \NS(K_H(A,v)), $$
		whose cokernel is a $p$-primary torsion group. 
		\item There is an isomorphism  $v^\perp \otimes K\rightarrow H^2_{\cris}(K_H(A,v)/W)_{K}$	as $F$-isocrystals. 
	\end{enumerate}
\end{proposition}

\begin{proof}
Everything is known in characteristic zero (\emph{cf.}~\cite{Mukai84}, \cite[Theorem 0.1]{Yo01}). For fields of positive characteristic, the proof is the same as in the case of K3 surfaces by the lifting techniques, with the following extra arguments for the smoothness of $K_H(A,v)$ and the existence of the symplectic form on it. For the smoothness, note that by Yoshioka \cite[P.840]{Yo01}, the base change of \eqref{albmap} by the multiplication-by-$(n+1)$ map on $A\times \widehat{A}$ becomes a trivial product. In other words, we have a cartesian diagram:
\begin{equation}
\label{eqn:YoshiokaTrivialization}
    \begin{tikzcd}
        K_H(A, v)\times A\times \widehat{A} \arrow[d, "p_{A\times\widehat{A}}"] \arrow[r, "\varphi"]& M_H(A, v) \arrow[d, "\alb"]\\
        A\times\widehat{A}  \arrow[r, "\times (n+1)"]& A\times \widehat{A}.
    \end{tikzcd}
\end{equation}
If $p\nmid (n+1)$, then the multiplication-by-$(n+1)$ map is \'etale, and so is $\varphi$ by base change. Therefore, the smoothness of $M_H(A, v)$ implies the smoothness of $K_H(A,v)$.  

As for the symplectic form, let $\omega$ be the symplectic form on $M_H(A, v)$ constructed by Mukai \cite{Mukai84}. If we know that $H^0(K_H(A,v), \Omega^1)=0$, then by the K\"unneth formula, there are 2-forms $\alpha$  and $\beta$ on $K_H(A, v)$ and $A\times \widehat{A}$ respectively, such that 
\begin{equation}
    \label{eqn:FromFactors}
    \varphi^*(\omega)=p_1^*(\alpha)+p_2^*(\beta),
\end{equation}
where $p_1$ and $p_2$ are projections from $K_H(A, v)\times A\times \widehat{A}$ to $K_H(A, v)$ and to $A\times \widehat{A}$ respectively. Note that $\alpha=\omega|_{K_H(A, v)}$. When $p\nmid (n+1)$, $\varphi$ is \'etale, hence $\varphi^*(\omega)$ is nowhere degenerate. By \eqref{eqn:FromFactors}, $\alpha$ (and $\beta$) must be nowhere degenerate as well, providing a symplectic form on $K_H(A, v)$.

As we do not have a proof of the vanishing of $H^0(K_H(A,v), \Omega^1)$ in general, let us provide the following workaround using lifting techniques.
By Proposition \ref{prop:LiftingAb}, one can lift $(A, H, v)$ to $(\mathcal{A}, \mathcal{H}, v)$ over a characteristic zero base $W'$. We have the relative moduli spaces of stable sheaves $\mathcal{M}:=\mathcal{M}_{\mathcal{H}}(\mathcal{A}, v)$ and $\mathcal{K}:=\mathcal{K}_{\mathcal{H}}(\mathcal{A}, v)$, together with a $W'$-morphism 
$\widetilde{\varphi}\colon \mathcal{K}\times_{W'} \mathcal{A}\times_{W'} \widehat{\mathcal{A}}\to \mathcal{M}$, in a way that the diagram \eqref{eqn:YoshiokaTrivialization} is the reduction of the analogous diagram over $W'$.
Mukai's construction \cite{Mukai84} gives a relative algebraic 2-form $\widetilde\omega\in  H^0(\mathcal{M}, \Omega^2_{\mathcal{M}/W'})$, which restricts to a symplectic form $\omega_\eta$ on the generic fiber and also to a symplectic form $\omega$ on the special fiber $M_H(A, v)$. 
When $p\nmid (n+1)$, $\widetilde{\varphi}$ is \'etale, hence $\widetilde{\varphi}^*(\widetilde\omega)$ gives again a symplectic form $\varphi_{\eta}^*(\omega_\eta)$ on the generic fiber $\mathcal{K}_\eta\times \mathcal{A}_\eta\times \widehat{\mathcal{A}}_\eta$ and a symplectic form $\varphi^*(\omega)$ on the special fiber $K_H(A, v)\times A\times \widehat{A}$. Since there is no non-zero algebraic 1-forms on $\mathcal{K}_\eta$, there exists an algebraic 2-form $\beta_\eta$ on $\mathcal{A}_\eta\times \widehat{\mathcal{A}}_\eta$ such that
\[\varphi_{\eta}^*(\omega_\eta)=p_1^*(\omega_\eta|_{\mathcal{K}_{\eta}})+p_2^*(\beta_\eta).\]
By reduction, there is an algebraic 2-form $\beta$ on $A\times \widehat{A}$, such that 
\begin{equation}
\label{eqn:FromFactors2} 
    \varphi^*(\omega)=p_1^*(\omega|_{K_{H}(A, v)})+p_2^*(\beta).
\end{equation}
As before, since $\varphi^*(\omega)$ is symplectic (by the fact that $\varphi$ is \'etale when $p\nmid (n+1)$), \eqref{eqn:FromFactors2} implies that $\omega|_{K_{H}(A, v)}$ is also symplectic. 
\end{proof}

By the same argument as in Lemma \ref{lemma:ConditionStar}, we can always reduce to the case where the polarization is general with respect to the Mukai vector, if the abelian surface is supersingular:
\begin{lemma}\label{lemma:ConditionStar:ab}
Let $A$ be a supersingular abelian surface defined over an algebraically closed field of characteristic $p>0$. If $v\in \tilde{H}(A)$ is coprime to $p$, then for any ample line bundle $H$, by performing a derived equivalence of tensoring with a line bundle, we have $\gcd(r, c_1\cdot H , s )=1$; hence $H$ is general with respect to $v$.
\end{lemma}

\begin{theorem}[{\cite[Theorem 0.2.11]{MMY11}}]\label{abwall}
Let $k$ be an algebraically closed field of characteristic $p>0$.  Let $A$ be a supersingular abelian surface over $k$. Let $v\in \tilde{H}(A)$ be a Mukai vector which is coprime to $p$. Let $H$ and $H'$ be two ample line bundles on $A$. Then $M_H(A,v)$ is quasi-liftably birational to $M_{H'}(A,v)$.
\end{theorem}

\begin{proof}   
The birational equivalence is exactly the assertion of \cite[Theorem 0.2.11]{MMY11}. To show that this birational equivalence is liftable, we apply  the lifting argument in Theorem \ref{K3wall}. The only thing one has to be careful is that according to Proposition \ref{prop:LiftingAb}, we can only lift supersingular abelian surface with at most two linearly independent line bundles to characteristic zero.  
	
Note that tensoring with a line bundle on $c_{1}$ would not change the moduli space, therefore, we can assume that $c_1=[H'']$ for some separably ample line bundle $H''$ after twisting some sufficiently  separable ample line bundle to $c_1$. Using the coprime condition, Lemma \ref{lemma:ConditionStar:ab}  then ensures $\gcd(r,c_1\cdot H,s)=1$. Thus we can lift the supersingular abelian surface with  line bundles $H$ and $H''$. In this case, the Mukai vector $v$ can be lifted as well.  So the  same proof in Theorem \ref{K3wall} shows that $M_{H}(A,v)$ is birationally equivalent to $M_{H''}(A,v)$ via some liftably birational map.  Similarly, there is a quasi-liftably birational map $M_{H'}(A,v)\dashrightarrow M_{H''}(A,v)$ and the assertion follows. 
\end{proof}

Now we can relate a moduli space of generalized Kummer type to generalized Kummer varieties via birational equivalences as below.  
\begin{theorem}\label{birKum}
Let $A$ be a supersingular abelian surface over an algebraically closed field $k$ of positive characteristic $p$. Let $v=(r, c_{1}, s)\in \widetilde{H}(A)$ be a Mukai vector coprime to $p$, with $r> 0$ and $v^{2}\geq 2$, then there is a birational map 
\begin{equation}\label{birationalKum}
M_{H}(A,v)\dashrightarrow  \Pic^0(A')\times A'^{[n+1]} 
\end{equation}
for some supersingular abelian surface $A'$ and $n=\frac{v^{2}}{2}-1$. Moreover, when $K_{H}(v)$ is smooth, there is a quasi-liftably birational equivalence 
\begin{equation}\label{birationalKumtype}
K_{H}(A,v)\dashrightarrow K_{n}(A').
\end{equation}
\end{theorem}

\begin{proof}
Let $c_1=c_1(L)$ for some line bundle $L\in \Pic(A)$ which we can assume to be ample by tensoring with a sufficiently ample line bundle. Let $E\in\Pic(A)$ be the line bundle which induces the elliptic fibration in Theorem \ref{ellstr:ab}.  One can easily prove that there exists a separable polarization $$H=L+nE\in \Pic(A)$$ for some $n\in \ZZ_{\geq 0}$ (up to a replacement of $L$ by $L+rL_1$ for some line bundle $L_1$). 
	

Now, one can lift the triple $(A, H, E)$ to characteristic $0$ by Proposition \ref{prop:LiftingAb}.  Using the same argument as in Theorem \ref{mainthm1}, one can deduce  \eqref{birationalKum}  by using Theorem \ref{abwall}, Theorem \ref{birationalK3} and a specialization argument. To see that $A'$ is supersingular, note that $M_{H}(A,v)$ and $\Pic^{0}(A')\times A'^{[n+1]}$, being birational, must have isomorphic Albanese varieties:
$$\widehat{A}\times A\simeq \widehat{A'}\times A',$$ which are supersingular. It follows that $A'$ is supersingular. The birational equivalence \eqref{birationalKum} yields the following  commutative diagram
\begin{equation*}
\xymatrix{
M_{H}(A,v) \ar@{-->}[r]_-{\sim} \ar[d]^{\alb}& \Pic^{0}(A')\times {A'}^{[n+1]}\ar[d]^{(\id, s)}\\
\Pic^{0}(A)\times A \ar[r]^{\simeq}& \Pic^{0}(A')\times A'
}
\end{equation*}
where the vertical maps are the isotrivial albanese fibrations, the bottom map is an isomorphism thanks to the general fact that a birational equivalence between two smooth projective varieties induces an isomorphism between their Albanese varieties. Now take any non-empty open subsets $U$ in $M_{H}(A,v)$ and $V$ in $\Pic^{0}(A')\times (A')^{[n+1]}$ which are identified under the birational equivalence \eqref{birationalKum}. By restricting to general fibers of the two isotrivial Albanese fibrations, this induces an isomorphism of a non-empty open subset of $K_{H}(A,v)$ to an open subset of $K_{n}(A')$, that is, a birational equivalence between them. 
\end{proof}

\begin{remark}\label{RCCbir}
\begin{enumerate}
\item  Similarly as in Remark \ref{rmk:Exception}, 	 the coprime assumption on $v$  is automatically satisfied if  $p\nmid (\frac{1}{2}\dim K_H(A,v)+1)$. 
\item The abelian surface $A'$  in  Theorem \ref{birKum} is derived equivalent to $A$.  When $A$ is the product of elliptic curves, we know that $A'$  has to be isomorphic to $A$ ({\it cf}.~\cite{HLT16}). In general, the supersingular abelian surface $A'$ is expected to be isomorphic to either $A$ or its dual $A^\vee$ ({\it cf}.~\cite{HLOY03}). 
\end{enumerate}
\end{remark}

\subsection{Proof of Theorem \ref{mainthm2}}	
Part $(i)$ follows directly from Proposition \ref{msheaf:ab} $(i)$ and $(iii)$. 
When $A$ is supersingular, as $X$ is quasi-liftably birationally equivalent to the generalized Kummer variety, we deduce that $\dim H^{2,0}(X)=1$, i.e.~$X$ is irreducible symplectic.

For the RCC conjecture: according to  the weak factorization theorem (\cite{MR2013783, MR1896232}), a liftably birational transformation can be viewed as the reduction of a sequence of blow-ups and blow-downs along smooth centers. As the exception divisors of the blow-ups and blow-downs are projective bundles over the center, the rational chain connectedness is preserved. Therefore, it follows from Theorem~\ref{birKum} that when $A$ is supersingular, $X$ is also rationally chain connected. 

Another way to prove the rational chain connectedness without using the lifting argument is to exploit the fact that $K_n(A')$ is ``unirational surface chain connected'', as is shown in the proof of Proposition \ref{prop:RCCKummer}. We leave the details to the interested readers.

Finally, the supersingular abelian motive conjecture follows from the combination of Theorem \ref{birKum}, Proposition \ref{prop:SSKummer} and Proposition \ref{prop:BirHK}. Therefore $X$ satisfies the fully supersingular Bloch--Beilinson conjecture \ref{conj:SSBB+} by Corollary \ref{cor:SSAV}. As is explained in \S\,\ref{subsect:Cycles}, to establish the supersingular Bloch--Beilinson--Beauville conjecture for $X$, it is enough to verify the section property conjecture \ref{conj:Section} for $X$. But this follows from that for $K_{n}(A)$, which is established in Proposition \ref{prop:SSKummer},  because Proposition \ref{prop:BirHK} provides an algebraic correspondence inducing an isomorphism of $\QQ$-algebras between $\CH^{*}(X)_{\QQ}$ and $\CH^{*}(K_{n}(A))_{\QQ}$ which sends Chern classes of $T_{X}$ to the corresponding ones of $T_{K_{n}(A)}$.

\subsection{A remark on the unirationality}
Theorem \ref{birKum} also provides us an approach to the unirationality conjecture for generalized Kummer type varieties $K_H(A, v)$. 

\begin{corollary}\label{cor-con3}
Notation and assumption are as in Theorem \ref{birKum}. If $A$ is supersingular, then $K_H(A, v)$ is unirational  if the generalized Kummer variety $K_{n}(E\times E)$ is unirational for some supersingular elliptic curve $E$, where $n=\frac{v^{2}}{2}-1$.
\end{corollary}
\begin{proof}
By Theorem \ref{birKum}, $K_{H}(A, v)$ is birational to $K_{n}(A')$ for some supersingular abelian surface $A'$. 
By \cite[Theorem 4.2]{Oo74}, for any supersingular elliptic curve $E$, $A'$ and $E\times E$ are isogenous. Therefore it suffices to see that for two isogenous abelian surfaces $A_{1}$ and $A_{2}$, the generalized Kummer varieties $K_{n}(A_{1})$ and $K_{n}(A_{2})$ are dominated by each other. To this end, we remark that $B_{1}:=\ker\left(A_{1}^{n+1}\xrightarrow{+} A_{1}\right)$ has an isogeny to $B_{2}:=\ker\left(A_{2}^{n+1}\xrightarrow{+} A_{2}\right)$ which is compatible with the natural $\mathfrak S_{n+1}$-actions. Therefore $K'_{n}(A_{1}):=B_{1}/{\mathfrak S_{n+1}}$ has a dominant map to $K'_{n}(A_{2}):=B_{2}/{\mathfrak S_{n+1}}$. Hence $K_{n}(A_{1})$ dominates $K_{n}(A_{2})$. By symmetry, they dominate each other.
\end{proof}

For  each $p$ and $n$,  Corollary \ref{cor-con3} allows us to  reduce the unirationality conjecture for moduli spaces of sheaves of generalized Kummer type to a very concrete question:  
\begin{question}\label{unikum}
Let $E$ be a supersingular elliptic curve. Is the generalized Kummer variety $K_n(E\times E)$ unirational\,? 
\end{question}
This question remains open even in the case $n=2$.

\bibliographystyle{abbrv}
\bibliography{Supersingular}

\end{document}